\documentclass[smallextended,numbook,runningheads]{svjour3}     

\smartqed  

\usepackage{amsmath, amssymb, amsfonts, mathrsfs}
\usepackage{graphicx}
\usepackage{color}
\usepackage{epstopdf} 
\usepackage{mathptmx}      
\usepackage{geometry}
\usepackage{appendix}
\usepackage[colorlinks,linkcolor=blue,anchorcolor=blue,citecolor=blue,CJKbookmarks=True]{hyperref}
\usepackage[american]{babel}
\usepackage{microtype}
\usepackage{ragged2e}

\usepackage[numbers,sort&compress]{natbib}

\usepackage{txfonts} 
\usepackage{marvosym} 

\newtheorem{Def}{Definition}[section]
\newtheorem{Ass}{Assumption}[section]
\newtheorem{Lem}{Lemma}[section]

\newtheorem{Theo}{Theorem}[section]
\newtheorem{Prop}{Proposition}[section]

\newtheorem{Rem}{Remark}[section]

\numberwithin{equation}{section}
\allowdisplaybreaks[4]

\newcommand{\hC}{\mathbb C}
\newcommand{\hD}{\mathbb D}
\newcommand{\hE}{\mathbb E}
\newcommand{\hN}{\mathbb N}
\newcommand{\hP}{\mathbb P}
\newcommand{\hR}{\mathbb R}

\newcommand{\rd}{\mathrm d}

\newcommand{\cC}{\mathcal C}

\newcommand{\cI}{\mathcal I}

\newcommand{\cL}{\mathcal L}
\newcommand{\cO}{\mathcal O}
\newcommand{\cS}{\mathcal S}
\newcommand{\cT}{\mathcal T}

\newcommand{\<}{\langle}
\renewcommand{\>}{\rangle}

\journalname{Potential Analysis}

\begin{document}

%
\title{Mittag--Leffler Euler integrator and large deviations for stochastic space-time fractional diffusion equations
}

\titlerunning{Mittag--Leffler Euler integrator and large deviations for FSPDEs}

\author{
Xinjie Dai \and
Jialin Hong \and
Derui Sheng\textsuperscript{\,\Letter}
}

\institute{
Xinjie Dai (\email{xinjie@smail.xtu.edu.cn}) \at
School of Mathematics and Statistics, Yunnan University, Kunming 650504, Yunnan, China 
\and
Jialin Hong (\email{hjl@lsec.cc.ac.cn}) \at
LSEC, ICMSEC, Academy of Mathematics and Systems Science, Chinese Academy of Sciences, Beijing 100190, China; \\
School of Mathematical Sciences, University of Chinese Academy of Sciences, Beijing 100049, China 
\and 
Derui Sheng (\email{sdr@lsec.cc.ac.cn}) \at
LSEC, ICMSEC, Academy of Mathematics and Systems Science, Chinese Academy of Sciences, Beijing 100190, China; \\
School of Mathematical Sciences, University of Chinese Academy of Sciences, Beijing 100049, China 
}

\date{Received: date / Accepted: date}

\maketitle

\begin{abstract}
Stochastic space-time fractional diffusion equations often appear in the modeling of the heat propagation in non-homogeneous medium. In this paper, we firstly investigate the Mittag--Leffler Euler integrator of a class of stochastic space-time fractional diffusion equations, whose super-convergence order is obtained by developing a helpful decomposition way for the time-fractional integral. Here, the developed decomposition way is the key to dealing with the singularity of the solution operator. Moreover, we study the Freidlin--Wentzell type large deviation principles of the underlying equation and its Mittag--Leffler Euler integrator based on the weak convergence approach. In particular, we prove that the large deviation rate function of the Mittag--Leffler Euler integrator $\Gamma$-converges to that of the underlying equation. 
\end{abstract}

\section{Introduction}
\label{sec.Intro}

In this paper, we study the following stochastic space-time fractional diffusion equation driven by fractionally integrated Gaussian noise with $\alpha \in (0,1)$, $\beta \in (0,1]$ and $\gamma \in [0,1]$ (see \cite{Kang2021IMA}):\ 
\begin{align} \label{eq.model}
\partial_t^{\alpha} X(t) + A^{\beta} X(t) = F(X(t)) + \cI {_t^{\gamma}} \dot{W}(t), \qquad \forall\, t \in (0,T] 
\end{align}
with $X(0) = X_0$. Here, $\partial_t^{\alpha}$ denotes the Caputo fractional time derivative defined by (see e.g., \cite{Podlubny1999}) 
\begin{align*}
\partial_t^{\alpha} X(t) = \frac{1}{\Gamma(1-\alpha)} \frac{\partial}{\partial t} \int_0^t (t-s)^{-\alpha} \big( X(s) - X(0) \big) \rd s
\end{align*}
with $\Gamma(\cdot)$ being the Gamma function, $A^\beta$ is the spectral fractional Laplacian with $A:=-\Delta$ being the negative Dirichlet Laplacian operator in some convex polygonal domain $\hD \subset \hR^{d}$ ($d \in \{1,2,3\}$),
and $\cI {_t^{\gamma}}$ is the Riemann--Liouville fractional time integral given by $\cI {_t^{\gamma}} \dot{W}(t) = \frac{1}{\Gamma(\gamma)} \int_0^t (t-s)^{\gamma-1} \dot{W}(s) \rd s$. Besides, the initial datum $X_0$ takes values in $H := L^2(\hD)$, $F:\ H \rightarrow H$ is a Borel measurable mapping, and $\{W(t)\}_{t\in[0,T]}$ denotes a cylindrical $Q$-Wiener process on some complete filtered probability space $(\Omega, \mathscr{F}, \{\mathscr{F}_t\}_{t\in[0,T]}, \hP)$, where the covariance operator $Q:H\to H$ is a bounded, linear, self-adjoint, positive semi-definite operator and not necessarily has finite trace.

\subsection{Physical background}

Stochastic fractional diffusion equations have been extensively applied to the modeling of some complicated noise systems with long memory or long-range correlations, such as the heat conduction in materials with thermal memory subject to noise factors, and particles fluctuating in viscoelastic fluids (see e.g., \cite{Chen2015SPA, MijenaNane2015, ChenHu2019}). In order to clarify the physical motivation of the model \eqref{eq.model}, we give two typical applications here.

Putting $\gamma=1-\alpha$ and applying $\partial_t^{1-\alpha}$ to \eqref{eq.model}, then it follows from $\partial_t^{\gamma}\mathcal I_t^{\gamma}\dot{W}(t)=\dot{W}(t)$ that
\begin{align}\label{eq:FKeq}
\partial_tX(t) + \partial_t^{1-\alpha} A^{\beta}X(t) = \partial_t^{1-\alpha}F(X(t)) + \dot{W}(t), \qquad \forall\, t \in (0,T].
\end{align}
The fractional power $A^{\beta}$ is the infinitesimal generator of some Markov process $Z$, which is obtained by killing the standard Brownian motion $B$ at the first exit time $\tau_{\mathbb{D}}:=\inf\{t>0:B(t)\notin \mathbb D\}$ of $B$ from the domain $\mathbb{D}$ and then subordinating the killed Brownian motion using the $\beta$-stable subordinator $T_t$, namely, 
\begin{align*}
Z(t) =
\begin{cases}
B(T_t),\quad &T_t<\tau_{\mathbb{D}},\\
\Theta,\quad &T_t\ge \tau_{\mathbb{D}},
\end{cases}
\end{align*}
where $\Theta$ is an isolated point serving as a cemetery. The Fokker--Planck equation corresponding to $Z(t)$ is given by (see e.g., \cite{NieDeng2022})
\begin{align*}
\partial_t u(t)+\partial_t^{1-\alpha} A^\beta u(t)+\frac{1}{\Gamma(\alpha)}t^{\alpha-1}A^\beta u(0)=0.
\end{align*}
When $u(0) = 0$ and the population of the particles is affected by the external source term depending on the density of the particles and the external Gaussian noise, we obtain the Fokker--Planck equation \eqref{eq:FKeq}, which corresponds to the model \eqref{eq.model} with $\gamma=1-\alpha$. 

For the general case where $\gamma$ is not necessarily $1-\alpha$, we provide another physical background of studying time-fractional \emph{stochastic partial differential equations} (SPDEs) with fractionally integrated additive noise, which is adapted from \cite{Chen2015SPA}. Let $u(t,x)$, $q(t,x)$ and $\vec{K}(t,x)$ denote the body temperature, energy and flux density, respectively. Then the relations 
\begin{equation*} 
\begin{split}
& \partial_t q(t,x) = -\textup{div} \vec{K}(t,x), \\
& q(t,x) = \sigma u(t,x), \qquad\qquad \vec{K}(t,x) = - \lambda \nabla u(t,x), \quad \sigma,\,\lambda > 0
\end{split}
\end{equation*}
yield the classical heat equation $\sigma \partial_t u = \lambda \Delta u$, which describes the heat propagation in homogeneous medium. For the non-homogeneous media, the relation $q(t,x) = \sigma u(t,x)$ commonly becomes 
\begin{align} \label{model.2}
q(t,x) = \int_0^t k(t-s) u(s,x) \rd s 
\end{align}
with some appropriate kernel $k(t)$. Typically, the power law kernel $t^{-\alpha}$ is chosen to describe that more recent past affects the present more. In some practical environments, external noises cannot be ignored. So it is necessary to consider the extension of \eqref{model.2}, for instance, 
\begin{align} \label{model.3}
q(t,x) = \int_0^t k(t-s) u(s,x) \rd s + \int_0^t \sigma(t-s) \rd W(s).
\end{align}
If $\lambda = 1$, $u(0,x) = 0$, $k(t) = \Gamma(1-\alpha)^{-1} t^{-\alpha}$ and $\sigma(t) = - \Gamma(\gamma + 1)^{-1} t^{\gamma}$, then differentiating \eqref{model.3} in $t$ reveals 
\begin{align*} 
-\textup{div} \vec{K}(t,x) = \frac{\partial}{\partial t} q(t,x) 
=\underbrace{ \frac{1}{\Gamma(1-\alpha)} \frac{\partial}{\partial t} \int_0^t (t-s)^{-\alpha} u(s,x) \rd s}_{=\, \partial_t^{\alpha} u(t,x)} -\underbrace{ \frac{1}{\Gamma(\gamma + 1)} \frac{\partial}{\partial t} \int_0^t (t-s)^{\gamma} \rd W(s)}_{=\, \cI {_t^{\gamma}} \dot{W}(t)},
\end{align*}
since $\frac{1}{\Gamma(\gamma+1)} \int_0^t (t-s)^{\gamma} \rd W(s) 
= \frac{1}{\Gamma(\gamma)} \int_0^t \int_0^{\tau} (\tau-s)^{\gamma-1} \rd W(s) \rd \tau$. 
Thus we can obtain the time-fractional stochastic heat equation $\partial_t^{\alpha} u -\Delta u = \cI {_t^{\gamma}} \dot{W}(t),$ which describes the heat conduction in non-homogeneous medium subject to noises (see e.g., \cite{Jin2019ESAIM, WuYan2020ANM, HuLi2022ANM}).

\subsection{Mathematical statement}
Let $\{(\lambda_k, \phi_k)\}_{k=1}^{\infty}$ be the sequence of eigenpairs of the negative Dirichlet Laplacian $A$ with $0<\lambda_1\leq \lambda_2\leq \cdots\le \lambda_k\leq \cdots$. In fact, $\{\phi_k\}_{k=1}^{\infty}$ forms an orthonormal basis of the separable Hilbert space $(H, \<\cdot, \cdot\>, \| \cdot \|)$. Set $\dot{H}^{\varpi}(\hD)$ or simply $\dot{H}^{\varpi}$ for any ${\varpi} \in \hR$ as a Hilbert space induced by the norm $\| \cdot \|_{\varpi} := \big( \sum_{k=1}^{\infty} \lambda_k^{\varpi} \< \cdot, \phi_k \>^2 \big)^{1/2}$. The spectral fractional Laplacian $A^{\beta}$ is defined by $A^{\beta} \varphi = \sum_{k=1}^{\infty} \lambda_k^{\beta} \<\varphi, \phi_k\>\phi_k$ for $\varphi \in \dot{H}^{2\beta}$. Let $\cL(H)$ be the space of bounded linear operators from $H$ to $H$ equipped with the usual operator norm $\| \cdot \|_{\cL(H)}$. Given separable Hilbert spaces $\mathcal{U}$ and $\mathcal{V}$,
denote by $\cL_2(\mathcal{U},\mathcal{V})$ the space of Hilbert--Schmidt operators $\cO:\ \mathcal{U} \rightarrow \mathcal{V}$ endowed with the Hilbert--Schmidt norm $\| \cO \|_{\cL_2(\mathcal{U},\mathcal{V})} := \big( \sum_{k=1}^{\infty} \| \cO \psi_k \|_{\mathcal{V}}^2 \big)^{1/2}$, where $\{\psi_k\}_{k=1}^\infty$ is an orthonormal basis of $\mathcal{U}$. Besides, we use the abbreviation $\cL_2(H):=\cL_2(H,H)$.

In recent years, there are a plenty of works on the mathematical theory of different types of fractional SPDEs (e.g., see \cite{AlloubaXiao2017, Chen2015SPA, ChenHu2019, ChenHu2022, MijenaNane2015} for the well-posedness and spatio-temporal regularity, and \cite{YanYin2018, Zhang2008JDE} for the \emph{large deviation principle} (LDP)). In order to study the mathematical theory of the model \eqref{eq.model}, we impose the following two assumptions. 

\begin{Ass} \label{ass.Noise}
Let $\alpha \in (0,1)$, $\beta \in (0,1]$, $\gamma \in [0,1]$ with $\alpha + \gamma > \frac{1}{2}$ and the covariance operator $Q$ satisfy 
\begin{align*}
\| A^{\frac{ r-\kappa}{2}} Q^{\frac{1}{2}} \|_{\cL_2(H)} < \infty \qquad \mbox{for some }\ r \in (0,\kappa],
\end{align*}
where $\kappa := \min\{ (\alpha+\gamma-\frac{1}{2})\frac{2\beta}{\alpha}, 2\beta \} - \varepsilon_0$ with $\varepsilon_0 > 0$ being arbitrarily small.
\end{Ass}

\begin{Ass} \label{ass.Lip1}
There exists a positive constant $L$ such that
\begin{align*}
\| F(\phi) \| \leq L( 1 + \| \phi \|) \quad \mbox{and} \quad \| F'(\phi)\psi \| \leq L\|\psi\|, \qquad \forall\, \phi, \psi \in H.
\end{align*}
\end{Ass}
\noindent Then the model \eqref{eq.model} admits a unique mild solution given by (see Theorem \ref{thm.ExisUniq}) 
\begin{align} \label{eq.mildSol}
X(t) = \cS_{1-\alpha}(t) X_0 + \int_0^t \cS_{0}(t-s) F(X(s)) \rd s + \int_0^t \cS_{\gamma}(t-s) \rd W(s), \qquad t \in [0,T].
\end{align}
Here, the solution operator $\cS_{\eta}(t): H \rightarrow H$ with $\eta \in [0,1]$ and $t > 0$ is given by
\begin{align} \label{eq.St}
\cS_{\eta}(t) \psi = t^{\alpha+\eta-1} \sum_{k=1}^{\infty} E_{\alpha,\alpha+\eta} (-\lambda_k^{\beta} t^{\alpha}) \<\psi, \phi_k\> \phi_k, \qquad \mbox{for } \psi \in H, 
\end{align} 
where $E_{a, b}(\cdot)$ with $a \in (0,1)$ and $b \in \hR$ is the Mittag--Leffler function (see Appendix \ref{sec.SolOpera}). For convenience, the mild solution \eqref{eq.mildSol} can also be represented as 
\begin{align} \label{eq.eqivaMild}
X(t) = \cS_{1-\alpha}(t) X_0 + \Upsilon_{F\circ X}(t) + \Lambda(t)
\end{align}
via the following two convolutions 
\begin{align} \label{eq.TwoConvo}
\Upsilon_{F\circ X}(t) := \int_0^t \cS_0(t-s) F(X(s)) \rd s, \qquad 
\Lambda(t) := \int_0^{t} \cS_{\gamma}(t-s) \rd W(s), \qquad \forall\, t \in [0,T].
\end{align}

As shown in Proposition \ref{prop.Regu}, the mild solution $X$ belongs to $L^\infty([0,T], L^p(\Omega,\dot H^r))$ and is H\"older continuous with order $\min\{ \alpha, \frac{\alpha r}{2\beta} + (\gamma-\frac{1}{2})^{+} \}$ in $L^p(\Omega,H)$ for $p\ge2$. These quantitive properties can be also extended to the controlled equation \eqref{eq.controlled-eq} and the skeleton equation \eqref{eq.skeleton-eq} associated with \eqref{eq.model}. On this basis, we obtain in Theorem \ref{th:LDPexact} the Freidlin--Wentzell type LDP of the model \eqref{eq.model} in the continuous function space $\mathcal C([0,T],H)$, by means of the weak convergence approach introduced in \cite{BD00}.

We further study the discrete-time simulations of the model \eqref{eq.model}, in view of the demand for realizing it on computers. A distinctive feature of the model \eqref{eq.model}, compared to the classical SPDE, is that its solution operator $\mathcal S_{\eta}(t)$ does not possess the semigroup property, and this brings additional challenges for developing efficient numerical methods and the related rigorous error analysis for the model \eqref{eq.model}. Let us introduce some existing results in this direction. The Gr\"unwald--Letnikov method and the L1 method as two time-stepping methods are successively studied in \cite{Jin2019ESAIM, WuYan2020ANM, HuLi2022ANM} for the linear case of \eqref{eq.model}. When $\alpha \in (\frac{1}{2},1)$, the model \eqref{eq.model} is also numerically solved in \cite{Kang2021IMA} by the Galerkin finite element method combined with the Wong--Zakai approximation, which mainly focuses on the spatial discretization. We also refer to \cite{Gunzburger2019MC, Kovacs2020SIAM, NieDeng2022} and references therein for the numerical studies of other variants of the model \eqref{eq.model}. Nevertheless, the rigorous error analysis of time-stepping methods of the model \eqref{eq.model} remains far from fully explored especially for the case $\alpha\in(0,\frac12)$ where the mild solution is rather rough.

Let $\{t_m = m h \}_{m = 0}^{M}$ with $h = T/M$ be a uniform partition of $[0,T]$, where $M \geq 2$ is some fixed integer. We discretize the model \eqref{eq.model} by the Mittag--Leffler Euler (MLE) integrator:
\begin{align} \label{eq.NumSol}
Y_m = \cS_{1-\alpha}(t_m) Y_0 + \sum_{j=0}^{m-1} \int_{t_j}^{t_{j+1}} \cS_{0}(t_m-s) F(Y_j) \rd s + \Lambda(t_m),
\end{align}
for $m = 1,2,\cdots,M$ and $Y_0 = X_0$. Notice that the stochastic convolution $\Lambda(\cdot)$ defined by \eqref{eq.TwoConvo} is a fractionally integrable Gaussian process and can be simulated on the grid points with a specified precision. When developing the error analysis of the MLE integrator, two main difficulties we encounter are \textsf{(Q1)}:\ the singularity of the solution operator $\cS_{0}(t)$; and \textsf{(Q2)}:\ the low regularity of the mild solution \eqref{eq.mildSol}. The difficulty \textsf{(Q1)} also appears in the error analysis of the MLE integrator for the finite-dimensional deterministic counterpart of the model \eqref{eq.model}, which has been well overcome in \cite{Garrappa2013} by using the asymptotic expansion $C_{0,0}+C_{0,1} t^\alpha+C_{0,2} t^{2\alpha}+C_{1,0} t + C_{1,1} t^{1+\alpha} + \cdots$ of the true solution. Unfortunately, an asymptotic expansion of this type seems unavailable for the mild solution \eqref{eq.mildSol} of \eqref{eq.model} due to the appearance of the noise source. Instead, we solve \textsf{(Q1)} by developing a useful decomposition way for the singular convolution $\Upsilon_{F\circ X}(\cdot)$ (see Proposition \ref{prop.UpsilonG}). In fact, the developed decomposition way also helps us to establish a higher-order regularity estimate for the mild solution \eqref{eq.mildSol} (see Proposition \ref{th.sinRegu}), which is the key to overcoming \textsf{(Q2)}. Together with the estimates of the stochastic convolution $\Lambda(\cdot)$, we establish the strong error analysis of the MLE integrator \eqref{eq.NumSol}, where the influence of the parameters $\alpha, \, \beta, \, \gamma$ and the regularity of the noise on the convergence rate is revealed. It turns out that the MLE integrator is super-convergent in the sense that its strong convergence order is higher than the temporal H\"older continuity exponent of the mild solution \eqref{eq.mildSol}; see Theorems \ref{th.MLEulerErrorLinear} and \ref{th.MLEulerErrorNonlinear} for the linear and nonlinear cases, respectively. In addition, we prove in Proposition \ref{th:LDPnum} that the Freidlin--Wentzell type LDP in $\mathcal C([0,T],H)$ also holds for the continuified MLE integrator, and in Theorem \ref{th:LDPratenum} that the large deviation rate function of the numerical solution $\Gamma$-converges to that of the exact solution.

The rest of the paper is organized as follows. Section \ref{sec.Prelim} introduces some notations and the regularity estimates of the singular convolution $\Upsilon_{F\circ X}(\cdot)$ and stochastic convolution $\Lambda(\cdot)$. In Section \ref{sec.wellpos}, we give the well-posedness, spatio-temporal regularity and the Freidlin--Wentzell type LDP for the model \eqref{eq.model}. In Section \ref{sec.MainResApp}, we present the strong convergence analysis and Freidlin--Wentzell type LDP for the MLE integrator of the model \eqref{eq.model}.

\section{Convolution estimates} 
\label{sec.Prelim}

Unless otherwise specified, the following notations will be used throughout this paper. Denote $a\vee b:=\max\{a,b\}$ and $a\wedge b:=\min\{a,b\}$ for $a,b\in\hR$, and in particular, $a^{+}: =a\vee 0$. Use $\varepsilon > 0$ to represent a sufficiently small constant and set $\inf\emptyset=\infty$ by convention. Denote by $C$ a generic constant and use $C(\cdot)$ if necessary to mention the parameters it depends on, which may change at each occurrence but are always independent of the stepsize $h$. 

In this section, we are devoted to analyzing the singular convolution 
\begin{align} \label{eq.Upsilon0t}
\Upsilon_G(t) := \int_0^t \cS_0(t-u) G(u) \rd u, \qquad \forall \, t \in [0,T]
\end{align}
and the stochastic convolution $\Lambda(t)$ defined by \eqref{eq.TwoConvo}. Here, $G:\ [0,T] \rightarrow L^p(\Omega,H)$ with some $p \geq 2$ is measurable, and we will take $G = F\circ X$ in Sections \ref{sec.wellpos} and \ref{sec.MainResApp}.

Considering the temporal regularity of $\Upsilon_G$, one usually uses the following decomposition 
\begin{align*} 
\Upsilon_{G}(t)- \Upsilon_{G}(s) 
= \int_{0}^{s} \big( \cS_{0}(t-u) - \cS_{0}(s-u) \big) G(u) \rd u 
+ \int_{s}^{t} \cS_{0}(t-u) G(u) \rd u 
\end{align*}
for $0 < s < t \leq T$.
As a result, under the assumption $\sup_{t \in [0,T]} \|G(t)\|_{L^p(\Omega,H)} < \infty$, by Lemma \ref{lem.Opera}, one obtains that for all $0 < s < t \leq T$, 
\begin{align} \label{eq.Way1Holder}
\left\| \Upsilon_{G}(t)- \Upsilon_{G}(s) \right\|_{L^p(\Omega,H)} 
&\leq \left\| \int_{0}^{s} \big( \cS_{0}(t-u) - \cS_{0}(s-u) \big) G(u) \rd u \right\|_{L^p(\Omega,H)} + \left\| \int_{s}^{t} \cS_{0}(t-u) G(u) \rd u \right\|_{L^p(\Omega,H)} \notag\\
&\leq C \int_{0}^{s} \left\| \cS_{0}(t-u) - \cS_{0}(s-u) \right\|_{\cL(H)} \rd u + C \int_{s}^{t} \left\| \cS_{0}(t-u) \right\|_{\cL(H)} \rd u \notag\\
&\leq C (t-s)^{\alpha}.
\end{align}
The above estimate reveals that the singularity of the solution operator $\cS_{0}$ will apparently affect the regularity of $\Upsilon_G$. In order to obtain a finer regularity estimate for the singular convolution $\Upsilon_G$, we introduce another decomposition way 
\begin{align}\label{eq.decomWay2}
\Upsilon_G(t) &= \underbrace{\int_0^t \cS_0(t-u) G(0) \rd u}_{=:\, \Phi(t)} + \underbrace{\int_0^t \cS_0(t-u) \big( G(u) - G(0) \big) \rd u}_{=:\, \Psi(t)}, \qquad \forall \, t \in [0,T].
\end{align}
As shown in Proposition \ref{prop.UpsilonG}, the advantage of the decomposition way \eqref{eq.decomWay2} lies in avoiding the singularity of $\cS_{0}$. To be specific, $\Phi(t) = \cS_1(t) G(0)$ can be bounded by using the fact that the operator $\cS_1(t)$ is non-singular, while the singularity of $\cS_{0}(t-u)$ in $\Psi(t)$ can be balanced by the H\"older continuity of $G$. Roughly speaking, one can regard the singular convolution $\Upsilon_G(t)$ as a time-fractional integral of order $\alpha$, and Proposition \ref{prop.UpsilonG} essentially indicates that the integral \eqref{eq.Upsilon0t} can lift the regularity of the integrand with the index $\alpha$.

\begin{Prop} 
\label{prop.UpsilonG}
Let $p \geq 2$ and $\Upsilon_G(t)$ be the convolution defined by \eqref{eq.Upsilon0t} with $G$ satisfying $\sup\limits_{t \in [0,T]} \|G(t)\|_{L^p(\Omega,H)} < \infty$. Then for the decomposition \eqref{eq.decomWay2}, one has the following statements$:$
\begin{enumerate}
\item[(1)] For all $0 < s < t \leq T$,
\begin{align*}
\| \Phi(t) - \Phi(s) \|_{L^p(\Omega,H)} \leq C ( t^{\alpha} - s^{\alpha} ).
\end{align*}

\item[(2)] If there exist constants $K > 0$, $\mu < 1$ and $\nu \in (-\alpha,1)$ such that for all $0 < u_1 < u_2 \leq T$,
\begin{align}\label{ass.G}
\|G(u_2) - G(u_1)\|_{L^p(\Omega,H)} \leq K u_1^{-\mu} ( u_2-u_1 )^{\nu},
\end{align}
then for all $0 < s < t \leq T$,
\begin{align*}
\| \Psi(t) - \Psi(s) \|_{L^p(\Omega,H)}
\leq C ( t^{\alpha} - s^{\alpha} ) + C s^{-\mu} (t-s)^{ \min\{\alpha+\nu-\varepsilon, 1\} }.
\end{align*}
As a result, for all $0 < s < t \leq T$,
\begin{align*}
\| \Upsilon_G(t) - \Upsilon_G(s) \|_{L^p(\Omega,H)}
\leq C ( t^{\alpha} - s^{\alpha} ) + C s^{-\mu} (t-s)^{ \min\{\alpha+\nu-\varepsilon, 1\} }.
\end{align*}
\end{enumerate}
Here, all the constants $C>0$ are independent of $t$ and $s$.
\end{Prop}

\begin{proof}
(1) By Lemma \ref{lem.Opera}(3),
\begin{align} \label{eq.S1diff0}
\| \cS_{1}(t) - \cS_{1}(s) \|_{\cL(H)}\le C\int_s^t u^{\alpha-1}\rd u
\leq C (t^{\alpha} - s^{\alpha}).
\end{align}
It follows from \eqref{eq.MLder} that
\begin{align*}
\Phi(t) - \Phi(s)
= \int_0^t \cS_0(t-u) G(0) \rd u - \int_0^s \cS_0(s-u) G(0) \rd u
= \big( \cS_1(t) - \cS_1(s) \big) G(0),
\end{align*}
which together with \eqref{eq.S1diff0} indicates
\begin{align*}
\| \Phi(t) - \Phi(s) \|_{L^p(\Omega,H)}
\leq \| \cS_1(t) - \cS_1(s) \|_{\cL(H)} \| G(0) \|_{L^p(\Omega,H)}
\leq C (t^{\alpha} - s^{\alpha}).
\end{align*}

(2) For convenience, denote $\tau := t - s > 0$. Then, by the change of variables,
\begin{align*}
\Psi(t) - \Psi(s)
&= \int_{-\tau}^{s} \cS_0(u+\tau) \big( G(s-u) - G(0) \big) \rd u - \int_{0}^{s} \cS_0(u) \big( G(s-u) - G(0) \big) \rd u \\
&= \underbrace{\int_{-\tau}^{s} \cS_0(u+\tau) \big( G(s) - G(0) \big) \rd u - \int_{0}^{s} \cS_0(u) \big( G(s) - G(0) \big) \rd u}_{=:\, \Psi_1} \\
&\quad + \underbrace{\int_{-\tau}^{0} \cS_0(u+\tau) \big( G(s-u) - G(s) \big) \rd u}_{=:\, \Psi_2} \\
&\quad + \underbrace{\int_{0}^{s} \big( \cS_0(u+\tau) - \cS_0(u) \big) \big( G(s-u) - G(s) \big) \rd u}_{=:\, \Psi_3}.
\end{align*}
It follows from the fact $\Psi_1 = \big( \cS_1(t) - \cS_1(s) \big) \big( G(s) - G(0) \big)$ and \eqref{eq.S1diff0} that
\begin{align*}
\| \Psi_1 \|_{L^p(\Omega,H)}
\leq \| \cS_1(t) - \cS_1(s) \|_{\cL(H)} \| G(s) - G(0) \|_{L^p(\Omega,H)} \leq C (t^{\alpha} - s^{\alpha}).
\end{align*}
Using Lemma \ref{lem.Opera}(1), \eqref{ass.G} and the Beta function $B(a,b) := \int_{0}^{1} u^{a-1} (1-u)^{b-1} \rd u$ for $a,b>0$ shows
\begin{align*}
\| \Psi_2 \|_{L^p(\Omega,H)}
&\leq \int_{-\tau}^{0} \| \cS_0(u+\tau) \|_{\cL(H)} \| G(s-u) - G(s) \|_{L^p(\Omega,H)} \rd u \\
&\leq C \int_{-\tau}^{0} (u+\tau)^{\alpha-1} s^{-\mu} (-u)^{\nu} \rd u \\
&\leq C s^{-\mu} \tau^{\alpha+\nu}.
\end{align*}
By \eqref{ass.G}, Lemma \ref{lem.Opera}(3) and the estimate $\int_{u}^{u+\tau} \omega^{\alpha-2} \rd \omega \leq C u^{\varepsilon-1-\nu} \tau^{ \min\{\alpha+\nu-\varepsilon, 1\} }$ with $u \in (0,T]$ and $\nu \in (-\alpha,1)$,
\begin{align*}
\| \Psi_3 \|_{L^p(\Omega,H)}
&\leq \int_{0}^{s} \| \cS_0(u+\tau) - \cS_0(u) \|_{\cL(H)} \| G(s-u) - G(s) \|_{L^p(\Omega,H)} \rd u \\
&\leq C \int_{0}^{s} \int_{u}^{u+\tau} \omega^{\alpha-2} \rd \omega \, (s-u)^{-\mu} u^{\nu} \rd u \\
&\leq C s^{\varepsilon-\mu} \tau^{ \min\{\alpha+\nu-\varepsilon, 1\} }.
\end{align*}
Therefore, by collecting these estimates at hand, we have
\begin{align*}
\| \Psi(t) - \Psi(s) \|_{L^p(\Omega,H)}\le \sum_{i=1}^3\| \Psi_i \|_{L^p(\Omega,H)} 
\leq C ( t^{\alpha} - s^{\alpha} ) + C s^{-\mu} (t-s)^{ \min\{\alpha+\nu-\varepsilon, 1\} }.
\end{align*}
The proof is completed. \qed
\end{proof}

\begin{Prop} \label{prop.Lambda}
Let $\Lambda(t)$ be the stochastic convolution defined by \eqref{eq.TwoConvo}. If Assumption \ref{ass.Noise} holds, then for any $\rho \in [r-2\kappa, r]$ and $p \geq 2$, there exists some positive constant $C$ such that for all $0 < s < t \leq T$, 
\begin{align*}
 \| \Lambda(t) \|_{L^p(\Omega,\dot{H}^{\rho})}
\leq C, \qquad
 \| \Lambda(t) - \Lambda(s) \|_{L^p(\Omega,\dot{H}^\rho)} \leq C (t-s)^{\min\{\alpha+\gamma-\frac{1}{2}-\frac{\alpha}{2\beta} (\kappa- r+\rho)^{+}, 1-\varepsilon\}}. 
\end{align*}
\end{Prop}

\begin{proof}
We only prove the second estimate since the first one is a byproduct. Note that 
\begin{align} \label{eq.LambdTwo}
\| \Lambda(t) - \Lambda(s) \|_{L^p(\Omega,\dot{H}^\rho)} 
\leq \Big\| \int_{s}^{t} \cS_{\gamma}(t-u) \rd W(u) \Big\|_{L^p(\Omega,\dot{H}^{\rho})} + \Big\| \int_{0}^{s} \cS_{\gamma}(t-u) - \cS_{\gamma}(s-u) \rd W(u) \Big\|_{L^p(\Omega,\dot{H}^{\rho})}.
\end{align}
Using Assumption \ref{ass.Noise} and Lemma \ref{lem.Opera}(1) indicates
\begin{align*}
\int_s^t \| A^{\frac{\rho}{2}} \cS_{\gamma}(t-u) Q^{\frac{1}{2}} \|_{\cL_2(H)}^2 \rd u 
&\leq \| A^{\frac{ r-\kappa}{2}} Q^{\frac{1}{2}} \|^2_{\cL_2(H)} \int_s^t \| A^{\frac{\kappa- r + \rho}{2}} \cS_{\gamma}(t-u) \|_{\cL(H)}^2 \rd u \notag\\
&\leq C (t-s)^{2( \alpha + \gamma - \frac{1}{2} - \frac{\alpha}{2\beta} (\kappa- r+\rho)^{+}) },
\end{align*}
which along with the Burkholder--Davis--Gundy inequality (see e.g., \cite{DaPrato2014Book}) yields
\begin{align} \label{eq.Lambda1sub}
\Big\| \int_{s}^{t} \cS_{\gamma}(t-u) \rd W(u) \Big\|_{L^p(\Omega,\dot{H}^{\rho})}
\leq C (t-s)^{\alpha + \gamma - \frac{1}{2} - \frac{\alpha}{2\beta} (\kappa- r+\rho)^{+}}.
\end{align}
Besides, Assumption \ref{ass.Noise} and Lemma \ref{le.DiffOpera} reveal
\begin{align*}
&\quad\ \Big( \int_0^s \| A^{\frac{\rho}{2}} ( \cS_{\gamma}(t-u) - \cS_{\gamma}(s-u) ) Q^{\frac{1}{2}} \|_{\cL_2(H)}^2 \rd u \Big)^{\frac{1}{2}} \nonumber\\
&\leq \| A^{\frac{ r-\kappa}{2}} Q^{\frac{1}{2}} \|_{\cL_2(H)} \Big( \int_0^s \| A^{\frac{\kappa- r + \rho}{2}} ( \cS_{\gamma}(t-u) - \cS_{\gamma}(s-u) ) \|_{\cL(H)}^2 \rd u \Big)^{\frac{1}{2}} \nonumber\\
&\leq
\begin{cases}
C (t-s)^{\frac{1}{2} - \frac{\alpha}{2\beta}(\kappa- r + \rho) }, & \mbox{if } \alpha + \gamma = 1, \\
C (t-s)^{\min\{\alpha+\gamma-\frac{1}{2}-\frac{\alpha}{2\beta} (\kappa- r+\rho)^{+}, 1-\varepsilon\}}, & \mbox{if } \alpha + \gamma \neq 1
\end{cases}\\
&\leq C (t-s)^{\min\{\alpha+\gamma-\frac{1}{2}-\frac{\alpha}{2\beta} (\kappa- r+\rho)^{+}, 1-\varepsilon\}}
\end{align*}
which in combination with the Burkholder--Davis--Gundy inequality gives
\begin{align} \label{eq.Lambda2sub}
 \Big\| \int_{0}^{s} \cS_{\gamma}(t-u) - \cS_{\gamma}(s-u) \rd W(u) \Big\|_{L^p(\Omega,\dot{H}^{\rho})} 
 \leq C (t-s)^{\min\{\alpha+\gamma-\frac{1}{2}-\frac{\alpha}{2\beta} (\kappa- r+\rho)^{+}, 1-\varepsilon\}}
\end{align}
Finally, combining the estimates \eqref{eq.LambdTwo}--\eqref{eq.Lambda2sub} completes the proof. \qed 
\end{proof}

\section{Stochastic space-time fractional diffusion equation}
\label{sec.wellpos}

This section gives the well-posedness, regularity and small noise asymptotic behavior of the stochastic space-time fractional diffusion equation \eqref{eq.model}. In particular, we present a singular-type temporal regularity estimate for the mild solution \eqref{eq.mildSol} by making full use of Proposition \ref{prop.UpsilonG}. In the strong convergence analysis of the MLE integrator \eqref{eq.NumSol}, the singular-type temporal regularity estimate is more instrumental than H\"older's regularity estimate, since the former has a higher exponent than the latter. 

\subsection{Well-posedness and regularity}

We begin with stating the existence and uniqueness of the mild solution \eqref{eq.mildSol}, whose proof can be found in Appendix \ref{sec.pfofwellpose}; see also \cite{Kang2021IMA} for the case of $\alpha \in (\frac{1}{2},1)$ and $\beta \in (\frac{1}{2},1]$.

\begin{Theo} [existence and uniqueness] \label{thm.ExisUniq}
Let $X_0 \in L^p(\Omega,H)$ with some $p \geq 2$ and Assumptions \ref{ass.Noise} and \ref{ass.Lip1} hold. Then the model \eqref{eq.model} admits the unique mild solution $X \in \cC([0,T], L^p(\Omega,H))$ given by \eqref{eq.mildSol}.
\end{Theo}

Next, we give the spatio-temporal regularity of the mild solution \eqref{eq.mildSol}, whose proof is mainly based on the estimates of solution operators (see Appendix \ref{sec.SolOpera}) and Proposition \ref{prop.Lambda}.

\begin{Prop} [regularity] \label{prop.Regu}
Let $p \geq 2$ and Assumptions \ref{ass.Noise}--\ref{ass.Lip1} hold.
\begin{enumerate}
\item[(1)] If $X_0 \in L^p(\Omega,\dot{H}^{r})$, then $X \in \cC([0,T], L^p(\Omega,\dot{H}^{r}))$.

\item[(2)] If $X_0 \in L^p(\Omega,\dot{H}^{2\beta})$, then there exists some positive constant $C$ such that
\begin{align} \label{Holder-1}
\| X(t) - X(s) \|_{L^p(\Omega,H)}
\leq C (t-s)^{ \min\{ \alpha, \frac{\alpha r}{2\beta} + (\gamma-\frac{1}{2})^{+} \} }, \qquad \forall \, 0\leq s < t \leq T.
\end{align}
\end{enumerate}
\end{Prop}

\begin{proof}
(1) Using \eqref{eq.mildSol}, Proposition \ref{prop.Lambda} with $\rho=r$, Lemma \ref{lem.Opera}(1) and Assumption \ref{ass.Lip1}, we obtain
\begin{align*}
\| X(t) \|_{L^p(\Omega,\dot{H}^{r})}
&\leq \| \cS_{1-\alpha}(t) \|_{\cL(H)} \| X_0 \|_{L^p(\Omega,\dot{H}^{r})} + \int_0^t \| A^{\frac{r}{2}} \cS_{0}(t-s) \|_{\cL(H)} \| F(X(s)) \|_{L^p(\Omega,H)} \rd s + C \\
&\leq C + C \int_0^t (t-s)^{\alpha-1-\frac{\alpha r}{2\beta}} (1 + \|X(s)\|_{L^p(\Omega,H)}) \rd s,\qquad\forall\,t\in[0,T],
\end{align*}
in which the norm $\| \cdot \|_{L^p(\Omega,H)} $ can be further bounded by $C \| \cdot \|_{L^p(\Omega,\dot{H}^{r})}$. Thus, applying the relation $r < 2\beta$ and the singular-type Gr\"onwall's inequality 
 yields the spatial regularity result. 

(2) By Lemma \ref{lem.Opera}(2) and $\| X_0 \|_{L^p(\Omega,\dot{H}^{2\beta})} < \infty$,
\begin{align*}
\| ( \cS_{1-\alpha}(t) - \cS_{1-\alpha}(s) ) X_0 \|_{L^p(\Omega,H)}
&\leq \| A^{-\beta} ( \cS_{1-\alpha}(t) - \cS_{1-\alpha}(s) ) \|_{\cL(H)} \| X_0 \|_{L^p(\Omega,\dot{H}^{2\beta})} \\ 
&\leq C \big( t^{\alpha} - s^{\alpha} \big) \leq C (t-s)^{\alpha}.
\end{align*}
It follows from Theorem \ref{thm.ExisUniq} and Assumption \ref{ass.Lip1} that $\sup_{t \in [0,T]} \| F(X(t)) \|_{L^p(\Omega,H)} < \infty$, which along with \eqref{eq.Way1Holder} implies 
\begin{align*}
\| \Upsilon_{F\circ X}(t) - \Upsilon_{F\circ X}(s) \|_{L^p(\Omega,H)} \leq C (t-s)^{\alpha}. 
\end{align*}
By Proposition \ref{prop.Lambda} with $\rho=0$ and $\kappa<\min\{ (\alpha+\gamma-\frac{1}{2})\frac{2\beta}{\alpha}, 2\beta \}$, 
\begin{align}\label{eq:Lamb}
\| \Lambda(t) - \Lambda(s) \|_{L^p(\Omega,H)} &\leq C (t-s)^{ \min\{ 1-\varepsilon, \frac{\alpha r}{2\beta} + (\gamma-\frac{1}{2})^{+} \} }.
\end{align}
Finally, collecting these estimates yields
\begin{align*}
\| X(t) - X(s) \|_{L^p(\Omega,H)}
&\leq \| ( \cS_{1-\alpha}(t) - \cS_{1-\alpha}(s) ) X_0 \|_{L^p(\Omega,H)} 
+ \| \Upsilon_{F\circ X}(t) - \Upsilon_{F\circ X}(s) \|_{L^p(\Omega,H)} + \| \Lambda(t) - \Lambda(s) \|_{L^p(\Omega,H)} \\
&\leq C (t-s)^{\alpha} + C (t-s)^{ \min\{ 1-\varepsilon, \frac{\alpha r}{2\beta} + (\gamma-\frac{1}{2})^{+} \} } \\
&\leq C (t-s)^{ \min\{ \alpha, \frac{\alpha r}{2\beta} + (\gamma-\frac{1}{2})^{+} \} }.
\end{align*}
The proof is completed. \qed
\end{proof}

\begin{remark}
%

In the following, we intent to discuss the regularity result in Proposition \ref{prop.Regu} under the condition
\begin{itemize}
\item[]
\begin{center}
(A1)
\textit{$X_0 \in L^p(\Omega,\dot{H}^{2\beta})$ and the noise $\dot{W}$ in \eqref{eq.model} is a space-time white noise (i.e., $Q=I$).}
\end{center}
\end{itemize}
When $Q=I$, Assumption \ref{ass.Noise} is equivalent to $\|A^{\frac{ r-\kappa}{2}}  \|^2_{\cL_2(H)}=\sum_{k=1}^\infty  \lambda_k^{r-\kappa}\approx \sum_{k=1}^\infty  k^{\frac{2}{d}(r-\kappa)} < \infty$ since $\lambda_k\approx k^{\frac{2}{d}}$ (see e.g., \cite[formula (5.2)]{CS05}). Here the series converges if and only if $r<-\frac{d}{2}+\kappa$. Then one can choose $r$ sufficiently near $-\frac{d}{2}+\kappa$ with $\kappa=\min\{ (\alpha+\gamma-\frac{1}{2})\frac{2\beta}{\alpha}, 2\beta \} - \varepsilon_0$ to guarantee that Assumption \ref{ass.Noise} holds.  Hence under (A1), the exponent $ \frac{\alpha r}{2\beta} + (\gamma-\frac{1}{2})^{+}$ in \eqref{Holder-1} is close to 
$ \frac{\alpha}{2\beta}(-\frac{d}{2}+\min\{ (\alpha+\gamma-\frac{1}{2})\frac{2\beta}{\alpha}, 2\beta \}) + (\gamma-\frac{1}{2})^{+}=- \frac{\alpha d}{4\beta}+ \alpha+\gamma-\frac{1}{2}$, which along with the Kolmogorov continuity theorem implies that the trajectories of $X$ are nearly 
$\min\{\alpha, - \frac{\alpha d}{4\beta}+ \alpha+\gamma-\frac{1}{2}\}$-H\"older continuous in time almost surely. If further $F\equiv0$, then by \eqref{eq:Lamb}, the trajectories of $X$ are nearly 
$\min\{1, - \frac{\alpha d}{4\beta}+ \alpha+\gamma-\frac{1}{2}\}$-H\"older continuous in time away from zero almost surely (see also e.g.,
\cite[Theorem 1]{ChenHu2022}). 
In particular, when $\gamma=1-\alpha$, the trajectories of $X$ are nearly $\min\{1, - \frac{\alpha d}{4\beta}+\frac{1}{2}\}$-H\"older continuous in time away from zero almost surely (see also e.g., \cite[formula (1.8)]{AlloubaXiao2017}). 
\end{remark}

With the temporal H\"older regularity estimate \eqref{Holder-1} at hand, we can prove the following singular-type temporal regularity estimate by repeatedly using Proposition \ref{prop.UpsilonG}.

\begin{Prop} [singular-type regularity] 
\label{th.sinRegu}
Let $X_0 \in L^p(\Omega,\dot{H}^{2\beta})$ for some $p \geq 2$ and Assumptions \ref{ass.Noise}--\ref{ass.Lip1} hold. Then
\begin{align*}
\| X(t) - X(s) \|_{L^p(\Omega,H)}
\leq C_{\varepsilon} s^{\varepsilon-\frac{1}{2}} (t-s)^{\min\{ \frac{\alpha r}{2\beta} + (\gamma-\frac{1}{2})^{+}, 1-\varepsilon \}}, \quad\ \ \forall\, 0 < s < t \leq T.
\end{align*}
\end{Prop}

\begin{proof}
For any $\alpha \in (0,1)$, $\rho \in [\alpha,1]$ and $0 < s < t\leq T$, one has
\begin{align} \label{esti.alpha-1}
t^{\alpha} - s^{\alpha}
= t^{\alpha-\rho} t^{\rho} - s^{\alpha-\rho} s^{\rho}
\leq s^{\alpha-\rho} t^{\rho} - s^{\alpha-\rho} s^{\rho}
\leq s^{\alpha-\rho} (t-s)^{\rho},
\end{align}
which implies that there exists some sufficiently small constant $\varepsilon > 0$ satisfying
\begin{align} \label{esti.alpha-2}
t^{\alpha} - s^{\alpha} \leq C s^{\varepsilon-\frac{1}{2}} (t-s)^{\min\{ \alpha+\frac{1}{2}-\varepsilon, 1 \}}.
\end{align}
From the proof of \eqref{Holder-1}, one can read
\begin{align} \label{Holder.conclu}
\| X(t) - X(s) \|_{L^p(\Omega,H)} \leq C(t^{\alpha} - s^{\alpha}) + \| \Upsilon_{F\circ X}(t) - \Upsilon_{F\circ X}(s) \|_{L^p(\Omega,H)} + C (t-s)^{ \min\{ 1-\varepsilon, \frac{\alpha r}{2\beta} + (\gamma-\frac{1}{2})^{+} \} }.
\end{align}
Based on \eqref{Holder-1} and Assumption \ref{ass.Lip1}, applying Proposition \ref{prop.UpsilonG} with $\mu = 0$, $\nu = \min\{ \alpha, \frac{\alpha r}{2\beta} + (\gamma-\frac{1}{2})^{+} \}$ and $G = F\circ X$ yields
\begin{align} \label{Holder.I2-2}
\| \Upsilon_{F\circ X}(t) - \Upsilon_{F\circ X}(s) \|_{L^p(\Omega,H)} \leq C ( t^{\alpha} - s^{\alpha} ) + C (t-s)^{\min\{2\alpha-\varepsilon, \alpha+\frac{\alpha r}{2\beta} + (\gamma-\frac{1}{2})^{+}-\varepsilon, 1 \}}.
\end{align}
Then, it follows from \eqref{Holder.conclu}, \eqref{Holder.I2-2} and \eqref{esti.alpha-2} that
\begin{align} \label{esti.Holder2}
\| X(t) - X(s) \|_{L^p(\Omega,H)}
&\leq C ( t^{\alpha} - s^{\alpha} ) + C (t-s)^{\min\{2\alpha-\varepsilon, \frac{\alpha r}{2\beta} + (\gamma-\frac{1}{2})^{+}, 1-\varepsilon \}} \nonumber\\
&\leq C s^{\varepsilon-\frac{1}{2}} (t-s)^{\min\{ \alpha+\frac{1}{2}-\varepsilon, 2\alpha-\varepsilon, \frac{\alpha r}{2\beta} + (\gamma-\frac{1}{2})^{+}, 1-\varepsilon \}} \nonumber\\
&\leq C s^{\varepsilon-\frac{1}{2}} (t-s)^{\min\{ 2\alpha-\varepsilon, \frac{\alpha r}{2\beta} + (\gamma-\frac{1}{2})^{+}, 1-\varepsilon \}},
\end{align}
where the last line is due to $ \alpha+\frac{1}{2}> \frac{\alpha r}{2\beta} + (\gamma-\frac{1}{2})^{+}$.
Based on \eqref{esti.Holder2} and Assumption \ref{ass.Lip1}, applying Proposition \ref{prop.UpsilonG} with $\mu = \frac{1}{2}-\varepsilon$, $\nu = \min\{ 2\alpha -\varepsilon, \frac{\alpha r}{2\beta} + (\gamma-\frac{1}{2})^{+}, 1-\varepsilon \}$ and $G = F\circ X$ yields
\begin{align} \label{Holder.I2-3}
\| \Upsilon_{F\circ X}(t) - \Upsilon_{F\circ X}(s) \|_{L^p(\Omega,H)}
\leq C ( t^{\alpha} - s^{\alpha} ) + C s^{\varepsilon-\frac{1}{2}} (t-s)^{\min\{ 3\alpha -2\varepsilon, \alpha + \frac{\alpha r}{2\beta} + (\gamma-\frac{1}{2})^{+} -\varepsilon, 1 \}}.
\end{align}
Then, it follows from \eqref{Holder.conclu}, \eqref{Holder.I2-3} and \eqref{esti.alpha-2} that
\begin{align} \label{esti.Holder3}
\| X(t) - X(s) \|_{L^p(\Omega,H)}
&\leq C ( t^{\alpha} - s^{\alpha} ) + C s^{\varepsilon-\frac{1}{2}} (t-s)^{\min\{ 3\alpha -2\varepsilon, \frac{\alpha r}{2\beta} + (\gamma-\frac{1}{2})^{+}, 1-\varepsilon \}} \nonumber\\
&\leq C s^{\varepsilon-\frac{1}{2}} (t-s)^{\min\{ \alpha+\frac{1}{2}-\varepsilon, 3\alpha-2\varepsilon, \frac{\alpha r}{2\beta} + (\gamma-\frac{1}{2})^{+}, 1-\varepsilon \}} \nonumber\\
&\leq C s^{\varepsilon-\frac{1}{2}} (t-s)^{\min\{ 3\alpha-2\varepsilon, \frac{\alpha r}{2\beta} + (\gamma-\frac{1}{2})^{+}, 1-\varepsilon \}}.
\end{align}
Repeating the derivation steps of \eqref{esti.Holder2} and \eqref{esti.Holder3} shows 
\begin{align*} 
\| X(t) - X(s) \|_{L^p(\Omega,H)}
\leq C s^{\varepsilon-\frac{1}{2}} (t-s)^{ \min\{ \frac{\alpha r}{2\beta} + (\gamma-\frac{1}{2})^{+}, 1-\varepsilon \} }.
\end{align*}
The proof is completed. \qed
\end{proof}

\subsection{Freidlin--Wentzell type LDP}
We begin with some preliminaries in the theory of large deviations. Let $\mathcal X$ be a Polish space. A real-valued function $I:\mathcal X\rightarrow[0,\infty]$ is called a rate function if it is lower semicontinuous, i.e., for each $a\in[0,\infty)$, the level set $I^{-1}([0,a]) := \{x \in \mathcal X \mid I(x) \in [0,a] \}$ is a closed subset of $\mathcal X$. Further, if for any $a\in[0,\infty)$, the level set $I^{-1}([0,a])$ is compact, then $I$ is called a good rate function. 
\begin{Def} \label{LDPdef}
A family of $\mathcal X$-valued random variables $\{\mathbf{X}^\epsilon\}_{\epsilon>0}$ defined on $(\Omega,\mathscr F,\mathbb P)$ is said to satisfy an LDP on $\mathcal X$ with the rate function $I$ if for every Borel set $U$ of $\mathcal X$,
\begin{align*}
-\inf_{x\in U^\circ} I(x)\le\liminf_{\epsilon\to 0}\epsilon\ln\mathbb P\{\mathbf{X}^\epsilon\in U\}\le\limsup_{\epsilon\to 0}\epsilon\ln\mathbb P\{\mathbf{X}^\epsilon\in U\}\leq-\inf_{x\in\bar U} I(U),
\end{align*}
where $U^\circ$ and $\bar U$ denote the interior and closure of $U$ in $\mathcal X$, respectively.
\end{Def} 

Set $H_0:=Q^{\frac12}(H)$ endowed with the inner product $( g,h)_0:=\langle Q^{-\frac12}g,Q^{-\frac12}h\rangle$ for $g,h\in H_0$ and the induced norm $|\cdot|_0:=\sqrt{(\cdot,\cdot)_0}$, where $Q^{-\frac12}$ is the pseudo inverse of $Q^{\frac12}$. Let $\{e_k\}_{k=1}^\infty$ form an orthonormal basis of $H_0$. Let $H_1:=H_0$ endowed with the norm $\|\cdot\|_{H_1} := \big(\sum_{k=1}^\infty\alpha_k^2(\cdot,e_k)_0^2\big)^{1/2}$, where $\{\alpha_k\}_{k=1}^\infty\subset(0,\infty)$ satisfies $\sum_{k=1}^\infty \alpha_k^2<\infty$.
Then the inclusion $\mathbb{J}:(H_0,|\cdot|_0)\to (H_1,\|\cdot\|_{H_1})$, $H_0\ni g\mapsto \mathbb J g=g\in H_1$ is a Hilbert--Schmidt embedding.
 It can be verified that $\mathbb J^* e_k=\alpha_k^2 e_k$ for $k\in\mathbb N_+$.
 In the following, we identify $g\in H_0$ with its image $\mathbb{J}(g)\in H_1$. Then the Wiener process $W$ can also be seen as an $H_1$ valued $Q_{\mathbb J}:={\mathbb J}{\mathbb J}^*$-Wiener process with the covariance operator $Q_{\mathbb J}\in\mathcal L(H_1)$ satisfying $\text{tr}(Q_{\mathbb J})<\infty$, and the paths of $W$ takes values in $\mathcal C([0,T],H_1)$ a.s. Notice that $\|Q_{\mathbb J}^{-\frac12}g\|_{H_1}=|g|_0$ for any $g\in H_0$.
 For the special case that $\text{tr}(Q)<\infty$ and $Q\phi_k=q_k \phi_k$ for $k\in\mathbb N_+$, we can take $e_k=\sqrt{q_k}\phi_k$, $\alpha_k=\sqrt{q_k}$ such that
 $(H_1,\|\cdot\|_{H_1})=(H,\|\cdot\|)$, $\mathbb J^*=Q$ and $Q_{\mathbb J}=Q$. 

In order to introduce the criterion for the LDP in \cite{BD00}, we define
\begin{align*}
&\mathcal{A}=\left\{\phi: \phi \text{ is an } H_0\text{-valued } \{\mathscr{F}_{t}\} \text{-predictable process such that } \int_{0}^{T}|\phi(s)|_{0}^{2} \mathrm{~d} s<\infty,~ \mathbb{P}\text{-a.s.} \right\},\\
&S_N=\left\{h \in L^{2}([0, T],H_0): \int_{0}^{T}|h(s)|_{0}^{2} \mathrm{~d} s \leq N\right\},\\
&\mathcal{A}_{N}=\left\{\phi \in \mathcal{A}: \phi(\omega) \in S_{N},~ \mathbb{P}\text{-a.s.} \right\},\quad N\in\mathbb N_+.
\end{align*}

\begin{Prop}[\cite{BD00}]\label{prop:criterion}
Let $\mathcal E$ be a Polish space. For each $\epsilon \ge 0$, suppose that $\mathcal{G}^{\epsilon}: \mathcal C([0, T], H_1) \rightarrow \mathcal{E}$ is a measurable map and the following conditions hold:\

\begin{enumerate}
\item[] \textbf{\textup{(C1)}} For every $N<\infty$, if $\{v^{\epsilon}\}_{\epsilon>0} \subset \mathcal{A}_{N}$ converges in distribution (as $S_{N}$-valued random elements) to $v$, then $$\mathcal{G}^{\epsilon}\left(W(\cdot)+\frac{1}{\sqrt{\epsilon}} \int_{0}^\cdot v^{\epsilon}(s) \mathrm{d} s\right)\rightarrow\mathcal{G}^{0}\left(\int_{0}^\cdot v(s) \mathrm{d} s\right)\quad \mbox{in distribution as } \epsilon\to 0;$$

\item[] \textbf{\textup{(C2)}} For every $N<\infty$, the set $\left\{\mathcal{G}^{0}\left(\int_{0}^\cdot v(s) \mathrm{d} s\right): v\in S_{N}\right\}$ is a compact subset of $\mathcal{E}$.
\end{enumerate}
Then $\{\mathcal{G}^{\epsilon}(W)\}_{\epsilon>0}$ satisfies an LDP on $\mathcal{E}$ as $\epsilon\to 0$ with a good rate function I given by
\begin{equation*}
I(x)=\inf_{\left\{v\in L^2([0,T],H_0),\,x=\mathcal{G}^0\left(\int_0^\cdot v(s)\rd s\right)\right\}}\frac12 \int_0^T|v(s)|_0^2\rd s,\quad x\in\mathcal E.
\end{equation*}
\end{Prop}

Next, we consider the following small perturbation to \eqref{eq.model}:\ 
\begin{align} \label{eq.model-small}
\partial_t^{\alpha} X^\epsilon(t) + A^{\beta} X^\epsilon(t) = F(X^\epsilon(t)) + \sqrt{\epsilon} \, \cI {_t^{\gamma}} \dot{W}(t), \qquad \forall\, t \in (0,T]
\end{align}
with $X^\epsilon(0) = X_0$. Define the functional $\mathcal G^\epsilon$ as the measurable map associating $W$ to the solution $X^\epsilon$ of \eqref{eq.model-small}, i.e., $X^\epsilon=\mathcal G^\epsilon(W)$ for $\epsilon>0$. For any control $v\in\mathcal A_N$ and $\epsilon>0$, the Girsanov theorem (see e.g., \cite[Theorem 10.14]{DaPrato2014Book}) indicates that $\widetilde W^{\epsilon,v}:=W+\frac{1}{\sqrt\epsilon}\int_0^\cdot v(s)\rd s$ is an $H_1$-valued $Q_{\mathbb J}$-Wiener process under $\widetilde\hP^{\epsilon,v}$, where
\begin{equation*}
\frac{\rd\widetilde\hP^{\epsilon,v}}{\rd\hP}:=\exp\left(-\frac{1}{\sqrt\epsilon}\int_0^T( v(s),\rd W(s))_0-\frac{1}{2\epsilon}\int_0^T|v(s)|_0^2\rd s\right).
\end{equation*}
Therefore $X^{\epsilon,v}:=\mathcal G^\epsilon(\widetilde W^{\epsilon,v})$ is the unique mild solution of \eqref{eq.model-small} under $\widetilde\hP^{\epsilon,v}$, with $(X^\epsilon,W)$ replaced by $(X^{\epsilon,v},\widetilde W^{\epsilon,v})$. Since $\hP$ is equivalent to $\widetilde\hP^{\epsilon,v}$, $X^{\epsilon,v}$ is also the unique mild solution of the following controlled equation
\begin{align} \label{eq.controlled-eq}
\partial_t^{\alpha} X^{\epsilon,v}(t) + A^{\beta} X^{\epsilon,v}(t) = F(X^{\epsilon,v}(t)) +\cI {_t^{\gamma}} v(t) + \sqrt{\epsilon}\cI {_t^{\gamma}} \dot{W}(t), \qquad \forall\, t \in (0,T]
\end{align}
with $X^{\epsilon,v}(0) = X_0$, under $\hP$. Heuristically, we observe that as $\epsilon\to0$, the controlled equation \eqref{eq.controlled-eq} formally reduces to the following skeleton equation
\begin{align} \label{eq.skeleton-eq}
\partial_t^{\alpha} Z^{v}(t) + A^{\beta} Z^{v}(t) = F(Z^{v}(t)) +\cI {_t^{\gamma}} v(t), \qquad \forall\, t \in (0,T] 
\end{align}
with $Z^v(0) = X_0$. Similarly to the proof of Theorem \ref{thm.ExisUniq}, one also has that \eqref{eq.skeleton-eq} admits a unique mild solution $Z^v=:\mathcal G^0\left(\int_0^\cdot v(s)\rd s\right)$ in $\mathcal C([0,T],H)$, where
\begin{align*}
Z^v(t) = \cS_{1-\alpha}(t) X_0 + \int_0^t \cS_{0}(t-s) F(Z^v(s)) \rd s + \int_0^t \cS_{\gamma}(t-s) v(s)\rd s, \qquad t \in [0,T].
\end{align*}

\begin{Theo}\label{th:LDPexact}
Let $X_0 \in L^p(\Omega,\dot{H}^{2\beta})$ for some $p \geq 2$ and Assumptions \ref{ass.Noise}--\ref{ass.Lip1} hold.
Then
 the solution $\{X^{\epsilon}\}_{\epsilon>0}$ to \eqref{eq.model-small} satisfies an LDP on $\mathcal C([0,T],H)$ as $\epsilon\to 0$ with the good rate function $I$ given by
\begin{equation*}
I(x)=\inf_{\left\{v\in L^2([0,T],H_0),\,x=Z^v\right\}}\frac12 \int_0^T|v(s)|_0^2\rd s,\quad x\in\mathcal C([0,T],H).
\end{equation*}
\end{Theo}

To facilitate the proof of Theorem \ref{th:LDPexact}, we prepare the following two lemmas.

\begin{Lem}\label{lem:tight-I}
Let Assumption \ref{ass.Noise} hold.
Define the operator $\Pi: L^2([0,T],H_0)\to \mathcal C([0,T],H)$ by 
\begin{align*}
\Pi(g)(t) := \int_0^t \cS_{\gamma}(t-s) g(s)\rd s, \qquad t\in[0,T],\quad g\in L^2([0,T],H_0).
\end{align*}
Then $\Pi$ is a compact operator, i.e., for any $N>0$,
$\Pi(S_N)$ is pre-compact in $\mathcal C([0,T],H)$.
\end{Lem}

\begin{proof}

For $\mu>0$ and $\theta\in(0,1]$, denote
\begin{equation}\label{eq:Ka}
\mathbb K_{a}^{\mu,\theta}:=\left\{y\in \mathcal C([0,T],\dot{H}^{\mu}):~\sup_{t\in[0,T]}\|y(t)\|_\mu+\sup_{s\neq t}\frac{\|y(t)-y(s)\|}{|t-s|^\theta}\le a\right\},\quad a>0.
\end{equation}
Then $\mathbb K_{a}^{\mu,\theta}$ is uniformly equi-continuous in $t\in[0,T]$ and the closure of $\mathbb K_{a}^{\mu,\theta}(t):=\{v(t):v\in\mathbb K_{a}^{\mu,\theta}\}$ is compact in $H$ for every $t\in[0,T]$ since the inclusion $\dot H^\mu\hookrightarrow H$ is a compact embedding for $\mu>0$. Therefore, $\mathbb K_{a}^{\mu,\theta}$ is pre-compact in $\mathcal C([0,T],H)$ due to the Arzel\`a--Ascoli theorem. 

By Cauchy--Schwartz's inequality and Lemma \ref{lem.Opera}(1), for any $\rho \in [r-2\kappa, r]$ and $g\in S_N$,
\begin{align}\label{eq:vepsilon}
\Big\| \int_{s}^{t} \cS_{\gamma}(t-u) g(u) \rd u\Big\|_{\rho}
&\le C\int_{s}^{t} \|A^{\frac{\kappa-r+\rho}{2}}\cS_{\gamma}(t-u) \|_{\cL(H)}\|A^{\frac{r-\kappa}{2}}Q^{\frac12}\|_{\cL_2(H)}|g(u)|_0 \rd u \nonumber\\
&\le C (t-s)^{ \alpha + \gamma - \frac{1}{2} - \frac{\alpha}{2\beta} (\kappa- r+\rho)^{+} } 
\end{align}
for any $0\le s<t\le T$. Similarly to \eqref{eq.Lambda2sub}, for any $\rho \in [r-2\kappa, r]$ and $g\in S_N$,
\begin{align}\label{eq:vepsilonts}
\Big\| \int_{0}^{s} (\cS_{\gamma}(t-u) - \cS_{\gamma}(s-u)) g(u)\rd u \Big\|_{\rho} 
\leq C (t-s)^{\min\{\alpha+\gamma-\frac{1}{2}-\frac{\alpha}{2\beta} (\kappa- r+\rho)^{+}, 1-\varepsilon\}}
\end{align}
for any $0\le s<t\le T$. Taking $\rho=r$ and $s=0$ in \eqref{eq:vepsilon}, as well as $\rho=0$ in \eqref{eq:vepsilon} and \eqref{eq:vepsilonts}, we have that for any $g\in S_N$,
\begin{align*}
\sup_{t\in[0,T]}\|\Pi(g)(t)\|_r+\sup_{s\neq t}\frac{\|\Pi(g)(t)-\Pi(g)(s)\|}{|t-s|^{\min\{ \frac{\alpha r}{2\beta} + (\gamma-\frac{1}{2})^{+}, 1-\varepsilon \}}}\le a 
\end{align*}
for some $a\in(0,\infty)$. This implies $\Pi(S_N)\subset \mathbb{K}_{a}^{r,\min\{\frac{\alpha r}{2\beta} + (\gamma-\frac{1}{2})^{+}, 1-\varepsilon \}}$, and the proof is completed. \qed
\end{proof}

Let $\tau:[0,1)\rightarrow[0,1)$ be a measurable function satisfying $\tau_\epsilon\rightarrow\tau_0 := 0$ as $\epsilon\to0$. 

\begin{Lem}\label{eq:weakcm}
Suppose that the assumptions of Theorem \ref{th:LDPexact} hold. 
Let $N<\infty$ and $\{v^{\epsilon}\}_{\epsilon>0} \subset \mathcal{A}_{N}$ converge in distribution (as $S_{N}$-valued random elements) to $v$. Then $X^{\tau_\epsilon,v^\epsilon}\to Z^{v}$ in distribution as $\epsilon\to 0$.
\end{Lem}

\begin{proof}
By \eqref{eq.controlled-eq}, for $t\in[0,T]$,
\begin{align}\label{eq:Xep}
X^{\tau_\epsilon,v^\epsilon}(t) &= \cS_{1-\alpha}(t) X_0 + \int_0^t \cS_{0}(t-s) F(X^{\tau_\epsilon,v^\epsilon}(s)) \rd s 
+ \int_0^t \cS_{\gamma}(t-s) v^\epsilon(s)\rd s+\sqrt{\tau_\epsilon}\Lambda(t),
\end{align}
where $\Lambda(t)$ is defined in \eqref{eq.TwoConvo}. 
In light of \eqref{eq:vepsilon}, for any $\rho \in [r-2\kappa, r]$ and $v^\epsilon\in\mathcal A_{N}$,
 \begin{align} \label{eq.intvup}
 \Big\| \int_{s}^{t} \cS_{\gamma}(t-u) v^\epsilon(u) \rd u\Big\|_{\rho}
 &\le C (t-s)^{ \alpha + \gamma - \frac{1}{2} - \frac{\alpha}{2\beta} (\kappa- r+\rho)^{+} },\quad\textup{a.s.},  
 \end{align}
 for any $0\le s<t\le T$. Hereafter, the constant $C$ is independent of $\epsilon\in[0,1]$ but may depend on $N$. 
Then similarly to Theorem \ref{prop.Regu}(1), it holds that 
\begin{equation}\label{eq:Xepsilonbound}
\| X^{\tau_\epsilon,v^\epsilon}(t) \|_{L^p(\Omega,\dot{H}^{r})}\le C,\qquad\forall\,t\in[0,T].
\end{equation}
From \eqref{eq:vepsilonts}, we obtain that for any $v^\epsilon\in\mathcal A_{N}$ and $\rho \in [r-2\kappa, r]$, 
\begin{align*}
\Big\| \int_{0}^{s} (\cS_{\gamma}(t-u) - \cS_{\gamma}(s-u)) v^\epsilon(u)\rd u \Big\|_{L^p(\Omega,\dot{H}^{\rho})}
\leq C (t-s)^{\min\{\alpha+\gamma-\frac{1}{2}-\frac{\alpha}{2\beta} (\kappa- r+\rho)^{+}, 1-\varepsilon\}}
\end{align*}
for any $0\le s<t\le T$. Moreover, similarly to Theorem \ref{prop.Regu}(2), we also have that for any $p>1$,
\begin{align} \label{Holder-Xepsilon}
\| X^{\tau_\epsilon,v^\epsilon}(t) - X^{\tau_\epsilon,v^\epsilon}(s)\|_{L^p(\Omega,H)}
\leq C (t-s)^{ \min\{ \alpha, \frac{\alpha r}{2\beta} + (\gamma-\frac{1}{2})^{+} \} }, \qquad \forall \, 0\leq s < t \leq T,
\end{align}
Here and after, we take $\theta\in(0,\min\{ \alpha, \frac{\alpha r}{2\beta} + (\gamma-\frac{1}{2})^{+} \})$. Then the Garsia--Rodemich--Ramsay lemma (see e.g., \cite[Theorem 2.1]{RY99}) together with \eqref{Holder-Xepsilon} yields 
\begin{align}\label{eq:XWtheta}
\hE \left[ \sup_{s\neq t} \frac{\|X^{\tau_\epsilon,v^\epsilon}(t)-X^{\tau_\epsilon,v^\epsilon}(s)\|}{|t-s|^\theta} \right] 
\leq C.
\end{align}

On the other hand, applying Proposition \ref{prop.Lambda} with $\rho=r$ and the Garsia--Rodemich--Ramsay lemma (see e.g., \cite[Theorem 2.1]{RY99}) yields that for any $p>1$,
\begin{equation}\label{eq:Lambdar}
\hE\left[\sup_{t\in[0,T]}\|\Lambda(t)\|_r^p\right]\le C.
\end{equation}
Introducing $U^{\tau_\epsilon,v^\epsilon}:=X^{\tau_\epsilon,v^\epsilon}-\sqrt{\tau_\epsilon}\Lambda$, then by Lemma \ref{lem.Opera}(1), Assumption \ref{ass.Lip1} and \eqref{eq.intvup}, we have
\begin{align*}
\|U^{\tau_\epsilon,v^\epsilon}(t)\|_r 
&\leq \| \cS_{1-\alpha}(t) \|_{\cL(H)} \| X_0 \|_r + \int_0^t \| A^{\frac{r}{2}} \cS_{0}(t-s) \|_{\cL(H)} \| F(X^{\tau_\epsilon,v^\epsilon}(s)) \| \rd s + C \\
&\leq C \| X_0 \|_r+ C \int_0^t (t-s)^{\alpha-1-\frac{\alpha r}{2\beta}} (1 + \|X^{\tau_\epsilon,v^\epsilon}(s)\|) \rd s + C \\
&\leq C \| X_0 \|_r+ \left(\int_0^t(t-s)^{q(\alpha-1-\frac{\alpha r}{2\beta})}\rd r\right)^{\frac1q}\left(\int_0^t\|X^{\tau_\epsilon,v^\epsilon}(s)\|^{q^\prime}\rd s\right)^{\frac{1}{q^\prime}}+C,
\end{align*}
for all $t\in(0,T]$, where $1/q+1/q^\prime=1$ and $q$ is large enough such that $q(\alpha-1-\frac{\alpha r}{2\beta})>-1$. Then taking supremum over $t\in[0,T]$ and taking expectations, from \eqref{eq:Xepsilonbound} and the assumption $X_0\in L^p(\Omega,\dot{H}^{2\beta})$, we deduce that
$$\hE\left[\sup_{t\in[0,T]}\|U^{\tau_\epsilon,v^\epsilon}(t)\|_r\right]\le C.$$
This in combination with \eqref{eq:Lambdar} and \eqref{eq:XWtheta} indicates
\begin{align*}
\hE\left[\sup_{t\in[0,T]}\|X^{\tau_\epsilon,v^\epsilon}(t)\|_r+
\sup_{s\neq t}\frac{\|X^{\tau_\epsilon,v^\epsilon}(t)-X^{\tau_\epsilon,v^\epsilon}(s)\|}{|t-s|^\theta}\right]\le C.
\end{align*}
Hence, further applying the Markov inequality and the tightness of $\mathbb K_a^{r,\theta}$ (see \eqref{eq:Ka}) leads to the tightness of $\{X^{\tau_\epsilon,v^\epsilon}\}_{\epsilon>0}$ in $\mathcal C([0,T],H)$. 

Since $S_N$ endowed with the weak topology of $L^2([0,T],H_0)$ is a compact Polish space \cite{BDM08}, $\{v^\epsilon\}_{\epsilon>0}$ is also tight in $L^2([0,T],H_0)$. As a result, $\{(X^{\tau_\epsilon,v^\epsilon},v^\epsilon)\}_{\epsilon>0}$ is tight in $\mathcal C([0,T],H)\otimes S_N$. In view of the Prokhorov theorem, for any sequence $\{\epsilon_n\}_{n\ge1}$ satisfying $\lim_{n\to \infty}\epsilon_n= 0$, there exists a subsequence, still denoted by $\{\epsilon_n\}_{n\ge1}$, such that $(X^{\tau_{\epsilon_n},v^{\epsilon_n}},v^{\epsilon_n})$ converges in distribution as $n\to \infty$. Therefore, further utilizing the Skorohod representation theorem (see e.g., \cite[Theorem 9.1.7]{KD01}) allows us to obtain another probability space $(\bar\Omega,\bar{\mathscr F},\bar\hP)$ carrying a random variable $(\bar{X},\bar{v})$ such that 
\begin{gather}\label{eq:vX}
v^{\epsilon_n}\overset{w}{\rightarrow}\bar{v}\quad \text{in}~ S_N,\qquad X^{\tau_{\epsilon_n},v^{\epsilon_n}}\rightarrow \bar{X}\quad\text{in}~\mathcal C([0,T],H),\qquad\text{in distribution}.
\end{gather}
Here, the notation $\overset{w}{\rightarrow}$ denotes the weak convergence.

It remains to prove that $\bar{X}\overset{d}{=} Z^v$, i.e., the law of $\bar{X}$ coincides with that of $Z^v$. To this end, introduce the map $\Psi_t: \mathcal C([0,T],H)\times S_N\to \mathbb R$ by
$$\Psi_t(\phi,f):=\Big\|\phi(t) - \cS_{1-\alpha}(t)\phi(0)- \int_0^t \cS_{0}(t-s) F(\phi(s)) \rd s - \int_0^t \cS_{\gamma}(t-s) f(s)\rd s\Big\|\wedge 1,$$
for $\phi\in\mathcal C([0,T],H)$ and $f\in S_N$. 
We firstly show that $\Psi_t$ is continuous. Assume that $\phi_n\to\phi$ in $\mathcal C([0,T],H)$ and $f_n\overset{w}{\rightarrow} f$ in $S_N$. Notice that $|\|a\|\wedge 1-\|b\|\wedge 1|\le \|a-b\|\wedge 1$ for any $a,b\in H$. Then by the Lipschitz continuity of $F$ and Lemma \ref{lem.Opera}(1), it holds that for any $t\in[0,T]$,
\begin{align*}
&\quad\ |\Psi_t(\phi_n,f_n)-\Psi_t(\phi,f)| \\ 
&\le \Big(\left\|\phi_n(t) -\phi(t)\right\| + \int_0^t \|\cS_{0}(t-s) \|_{\cL(H)}\left\|F(\phi_n(s))-F(\phi(s))\right\| \rd s + \left\|\Pi(f_n)(t)-\Pi(f)(t)\right\|\Big)\wedge 1\\
&\le \left( C\left\|\phi_n -\phi\right\|_{\mathcal C([0,T],H)} + \left\|\Pi(f_n)-\Pi(f)\right\|_{\mathcal C([0,T],H)}\right)\wedge 1,
\end{align*}
where $\Pi$ is defined in Lemma \ref{lem:tight-I}.
In light of the compactness of $\Pi$ (see Lemma \ref{lem:tight-I}) and $f_n\overset{w}{\rightarrow} f$ in $S_N$, we know that $\left\|\Pi(f_n)-\Pi(f)\right\|_{\mathcal C([0,T],H)}\to 0$ as $n\to\infty$. Thus $\Psi_t$ is continuous. Since $\Psi_t$ is continuous and bounded, it follows from \eqref{eq:vX} that
$\hE[\Psi_t(X^{\tau_{\epsilon_n},v^{\epsilon_n}},v^{\epsilon_n})]\to \hE^{\bar\hP}[\Psi_t(\bar{X},\bar{v})]$ as $n\to\infty$,
where $\hE^{\bar\hP}$ denotes the expectation with respect to $\bar\hP$.
In addition, the definition of $\Psi_t$, \eqref{eq:Xep} and Proposition \ref{prop.Lambda} imply
\begin{align*}
\hE[\Psi_t(X^{\tau_{\epsilon_n},v^{\epsilon_n}},v^{\epsilon_n})]=\left(\sqrt{\tau_{\epsilon_n}}\, \hE\left[\|\Lambda(t)\|\right]\right)\wedge 1\le \left(C\sqrt{\tau_{\epsilon_n}}\right)\wedge1\to 0, \qquad \mbox{as } n\to \infty,
\end{align*}
which leads to $\hE^{\bar\hP}[\Psi_t(\bar{X},\bar{v})]=0$, and thus for any $t\in[0,T]$,
\begin{align*}
\bar{X}(t) = \cS_{1-\alpha}(t) X_0 + \int_0^t \cS_{0}(t-s) F(\bar{X}(s)) \rd s + \int_0^t \cS_{\gamma}(t-s) \bar{v}(s)\rd s, \qquad \bar\hP\text{-a.s.}
\end{align*}
Then the uniqueness of the skeleton equation \eqref{eq.skeleton-eq} leads to $\bar{X}=\mathcal G^0\left(\int_0^\cdot \bar v(s)\rd s\right),$ $\bar\hP$\text{-a.s.} Since $\{v^{\epsilon_n}\}_{n\ge1}$ weakly converges to both $v$ and $\bar v$ in distribution, we have $\bar v\overset{d}=v$. Thus $\bar{X}=\mathcal G^0\left(\int_0^\cdot \bar v(s)\rd s\right)\overset{d}=\mathcal G^0\left(\int_0^\cdot v(s)\rd s\right)=Z^v,$ as required. The proof is completed. \qed 
\end{proof}

\textbf{Proof of Theorem \ref{th:LDPexact}.} In view of Proposition \ref{prop:criterion}, it suffices to show that \textbf{(C1)} and \textbf{(C2)} are fulfilled with $\mathcal E$ there replaced by $\mathcal C([0,T],H)$. The condition \textbf{(C1)} comes from Lemma \ref{eq:weakcm} by taking $\tau_\epsilon=\epsilon$ for $\epsilon\in[0,1)$. On the other hand, by taking $\tau_\epsilon=0$ for all $\epsilon\in[0,1)$ in Lemma \ref{eq:weakcm} and the uniqueness of the solution to the skeleton equation \eqref{eq.skeleton-eq}, we obtain that the mapping $S_N\ni v\mapsto Z^v\in \mathcal C([0,T],H)$ is continuous. This proves that the condition \textbf{(C2)} holds since $S_N$ endowed with the weak topology is a compact space. The proof is completed. \hfill$\Box$

\section{Mittag--Leffler Euler integrator }
\label{sec.MainResApp}

In this section, we develop the strong convergence analysis of the MLE integrator \eqref{eq.NumSol} for the model \eqref{eq.model}. The strong convergence provides a pathwise approximation of individual sample paths of the exact solution, and has been frequently applied in the design of the multilevel Monte Carlo algorithm. In addition, we also present the Freidlin--Wentzell type LDP for the continuified MLE integrator and the $\Gamma$-convergence analysis for the associated large deviation rate function.

\subsection{Convergence rate}

Let $\{X(t)\}_{t\in[0,T]}$ and $\{Y_m\}_{m=1}^M$ be the mild solution \eqref{eq.mildSol} of the model \eqref{eq.model} and its numerical solution given by the MLE integrator \eqref{eq.NumSol}, respectively. Based on the temporal H\"older continuity estimate \eqref{Holder-1}, a preliminary strong convergence order $\min\{ \alpha, \frac{\alpha r}{2\beta} + (\gamma-\frac{1}{2})^{+}\}$ can be obtained for the MLE integrator. However, since the numerical solution $\{Y_m\}_{m=1}^M$ uses the accurate information of the stochastic convolution $\Lambda(t)$, a higher strong convergence order than $\min\{ \alpha, \frac{\alpha r}{2\beta} + (\gamma-\frac{1}{2})^{+}\}$ can be expected, which is exactly the main aim of the following two theorems.

\begin{Theo} \label{th.MLEulerErrorLinear}
Let $X_0 \in L^2(\Omega,\dot{H}^{2\beta})$. Suppose that Assumption \ref{ass.Noise} holds and the mapping $F$ in \eqref{eq.model} is linear. Then there exists some positive constant $C = C(\alpha,\beta,\gamma,r,T,\varepsilon)$ such that for all $2 \leq m \leq M$,
\begin{align*}
\| X(t_m) - Y_m \|_{L^2(\Omega,H)}
\leq C t_m^{2\alpha-1} h^{ \min\{ \alpha + \frac{\alpha r}{2\beta} + (\gamma-\frac{1}{2})^{+}-\varepsilon, \, \alpha+\gamma -\varepsilon, \, 1 \} }, 
\end{align*}
where $\varepsilon> 0$ can be arbitrarily small.
\end{Theo}

Further, we weaken the restriction of $F$ and obtain the following theorem.

\begin{Theo} \label{th.MLEulerErrorNonlinear}
Let $X_0 \in L^4(\Omega,\dot{H}^{2\beta})$ with $\beta \in (\frac{1}{2},1]$. Suppose that Assumptions \ref{ass.Noise}--\ref{ass.Lip1} hold, the mapping $F$ in \eqref{eq.model} is twice differentiable and there exist $\delta \in [1,2\beta)$ and $\zeta \in [1, 2\beta)$ such that
\begin{align}
& \| F'(\varphi_1)\varphi_2 \|_{-\delta}
\leq L ( 1 + \| \varphi_1 \|_{r} ) \| \varphi_2 \|_{-r},
\!\!\!\!\!\!\!\!\!\!\!\!\!\!\!\!\!\!\!\!\!\!\!\!\!\!\!\!\!\!\!\!\!\!\!\!
&& \forall\, \varphi_1 \in \dot{H}^{r}, \, \varphi_2 \in \dot{H}^{-r}, \label{ass.1stF}\\
& \| F''(u) \<\phi, \psi\> \|_{-\zeta}
\leq L \| \phi \| \, \| \psi \|,
\!\!\!\!\!\!\!\!\!\!\!\!\!\!\!\!\!\!\!\!\!\!\!\!\!\!\!\!\!\!\!\!\!\!\!\!
&& \forall\, \phi, \, \psi \in H. \label{ass.2ndF}
\end{align}
Then there exists some positive constant $C = C(\alpha,\beta,\gamma,r,\delta,\zeta,L,T,\varepsilon)$ such that for all $2 \leq m \leq M$,
\begin{align*}
\| X(t_m) - Y_m \|_{L^2(\Omega,H)}
\leq
\begin{cases}
C t_m^{\alpha-1-\frac{\alpha\zeta}{2\beta}} h^{ \frac{\alpha r}{\beta} + \min\{ \frac{1}{2} - \frac{\alpha\kappa}{2\beta}, 2(\gamma-\frac{1}{2})^{+} \} }, & \mbox{if } \alpha + \gamma = 1, \\
C t_m^{\alpha-1-\frac{\alpha\zeta}{2\beta}} h^{\min\{\frac{\alpha}{2\beta}\min\{\kappa,2r\} + (\gamma-\frac{1}{2})^{+}, 1-\varepsilon\}}, & \mbox{if } \alpha + \gamma \neq 1.
\end{cases}
\end{align*}
\end{Theo}
The conditions \eqref{ass.1stF} and \eqref{ass.2ndF} are common in the error analysis of the exponential-type integrators for nonlinear (fractional) SPDEs. In particular, the mapping $F$ can be taken to be a Nemytskii operator defined by $F(U)(x) = f(U(x))$, where $U:\ \hD \rightarrow \hR$, and $f:\ \hR \rightarrow \hR$ is a smooth function with bounded derivatives of orders $1$ and $2$; see e.g., \cite{WangQi2015, Kovacs2020SIAM} for the concrete examples. 

Before proving Theorems \ref{th.MLEulerErrorLinear} and \ref{th.MLEulerErrorNonlinear}, we give several remarks on the strong convergence orders therein to identify the motivations and contributions of our work.

\begin{Rem}
Under the conditions \eqref{ass.1stF} and \eqref{ass.2ndF}, it is shown in \cite{Kovacs2020SIAM} that when applying the MLE integrator to the following stochastic Volterra integro-differential equation 
$$\partial_t u(t)+\frac{1}{\Gamma(\alpha)}\int_0^t(t-s)^{\alpha-1}Au(s)\rd s= F(u(t))+\dot{W}(t),\qquad t\in[0,T]$$
with suitable initial value $u(0)$, the corresponding strong convergence order is exactly twice of the temporal H\"older continuity exponent of the mild solution.
The same phenomenon also occurs in Theorem \ref{th.MLEulerErrorNonlinear} when $\alpha\in(\frac{1}{2},1)$ and $\gamma=1-\alpha$, since $\frac{\alpha r}{\beta} + \min\{ \frac{1}{2} - \frac{\alpha\kappa}{2\beta}, 2(\gamma-\frac{1}{2})^{+} \} = \frac{\alpha r}{\beta}$ and $\min\{\alpha, \frac{\alpha r}{2\beta} + (\gamma-\frac{1}{2})^{+} \} = \frac{\alpha r}{2\beta}$ at this time. On the other hand, when $\alpha$ is sufficiently small such that $ \alpha<\frac{\alpha r}{2\beta} + (\gamma-\frac{1}{2})^{+}$, the strong convergence order in Theorem \ref{th.MLEulerErrorNonlinear} may be higher than twice the temporal H\"older continuity exponent $\alpha$ of the mild solution \eqref{eq.mildSol}. However, this does not hold when the mild solution has relatively high temporal H\"older regularity since the strong convergence order of the MLE integrator will never be greater than $1$. 


From Theorems \ref{th.MLEulerErrorLinear} and \ref{th.MLEulerErrorNonlinear}, it can be observed that there is an order gap between the linear case and the nonlinear case. This phenomenon also appears when applying the numerical method \eqref{eq.NumSol} to SPDEs driven by the space-time white noise, which corresponds to the case $(\alpha,\beta,\gamma,r) = (1,1,0,\frac{1}{2}-2\varepsilon)$. More precisely, the existing results show that the corresponding strong convergence orders are respectively $1-\varepsilon$ and $\frac{1}{2}-\varepsilon$ for the linear and nonlinear cases; see e.g., \cite{Jentzen2009ProcA, WangQi2015}. It seems that the strong convergence order in Theorem \ref{th.MLEulerErrorNonlinear} for the nonlinear case may be not sharp. We leave this problem as an open problem, and will pursue this line of research in future work.
\end{Rem}

In order to prove Theorems \ref{th.MLEulerErrorLinear} and \ref{th.MLEulerErrorNonlinear}, we first give the strong error decomposition.
In view of \eqref{eq.mildSol} and \eqref{eq.NumSol}, the $L^2(\Omega,H)$-norm of the error of the numerical solution reads
\begin{align*}
\| X(t_m) - Y_m \|_{L^2(\Omega,H)}
&\leq \Big\| \sum_{j=0}^{m-1} \int_{t_j}^{t_{j+1}} \cS_{0}(t_m-s) ( F(X(s)) - F(Y_j) ) \rd s \Big\|_{L^2(\Omega,H)} \nonumber\\
&\leq \Big\| \sum_{j=0}^{m-1} \int_{t_j}^{t_{j+1}} \cS_{0}(t_m-s) ( F(X(t_j)) - F(Y_j) ) \rd s \Big\|_{L^2(\Omega,H)} \nonumber\\
&\quad + \Big\| \sum_{j=0}^{m-1} \int_{t_j}^{t_{j+1}} \cS_{0}(t_m-s) ( F(X(s)) - F(X(t_j)) ) \rd s \Big\|_{L^2(\Omega,H)}.
\end{align*}
For the integrand in the second term of the right hand side, it follows from Taylor's expansion that
\begin{align*}
F(X(s)) - F(X(t_j)) = F'(X(t_j)) (X(s) - X(t_j)) + R_{F,j}(s),
\end{align*}
where $R_{F,j}(s) := \int_{0}^{1} F''(X(t_j) + \theta (X(s) - X(t_j))) \<X(s) - X(t_j), X(s) - X(t_j)\> (1-\theta) \rd \theta$. Then, applying \eqref{eq.eqivaMild} shows
\begin{align} \label{eq.err5term}
\| X(t_m) - Y_m \|_{L^2(\Omega,H)} \leq J_{1} + J_{2} + J_{3} + J_{4} + J_{5}
\end{align}
with
\begin{align*}
J_{1} &:= \Big\| \sum_{j=0}^{m-1} \int_{t_j}^{t_{j+1}} \cS_{0}(t_m-s) ( F(X(t_j)) - F(Y_j) ) \rd s \Big\|_{L^2(\Omega,H)} , \\
J_{2} &:= \Big\| \sum_{j=0}^{m-1} \int_{t_j}^{t_{j+1}} \cS_{0}(t_m-s) F'(X(t_j)) ( \cS_{1-\alpha}(s) - \cS_{1-\alpha}(t_j) ) X_0 \rd s \Big\|_{L^2(\Omega,H)} , \\
J_{3} &:= \Big\| \sum_{j=0}^{m-1} \int_{t_j}^{t_{j+1}} \cS_{0}(t_m-s) F'(X(t_j)) \big( \Upsilon_{F\circ X}(s) - \Upsilon_{F\circ X}(t_j) \big) \rd s \Big\|_{L^2(\Omega,H)}, \\
J_{4} &:= \Big\| \sum_{j=0}^{m-1} \int_{t_j}^{t_{j+1}} \cS_{0}(t_m-s) F'(X(t_j)) \big( \Lambda(s) - \Lambda(t_j) \big) \rd s \Big\|_{L^2(\Omega,H)}, \\
J_{5} &:= \Big\| \sum_{j=0}^{m-1} \int_{t_j}^{t_{j+1}} \cS_{0}(t_m-s) R_{F,j}(s) \rd s \Big\|_{L^2(\Omega,H)}.
\end{align*}
In fact, the strong error decomposition \eqref{eq.err5term} is more or less standard. However, due to the singularity and non-Markovian property of the solution operator as well as the low regularity of the mild solution, it is challenging to delicately deal with the error terms $J_3$, $J_4$ and $J_5$ for the purpose of proving Theorems \ref{th.MLEulerErrorLinear} and \ref{th.MLEulerErrorNonlinear}. Next, we estimate separately the error $\| X(t_m) - Y_m \|_{L^2(\Omega,H)}$ in the linear and nonlinear cases.


\emph{Proof of Theorem \ref{th.MLEulerErrorLinear}.} For any $\rho \in [0,1]$, $s \in [t_j, t_{j+1}]$ with $j = 1,\cdots,m-1$, we have 
\begin{align} \label{esti.tj->s}
t_j^{-\rho} \leq 2^{\rho} t_{j+1}^{-\rho} \leq 2^{\rho} s^{-\rho},
\end{align}
which together with \eqref{esti.alpha-1} and the Beta function implies
\begin{align} \label{eq.usetoJ2}
\sum_{j=0}^{m-1} \int_{t_j}^{t_{j+1}} (t_m-s)^{\alpha-1} \big( s^{\alpha} - t_j^{\alpha} \big) \rd s
\leq t_{m-1}^{2\alpha-1} h + C \int_{t_1}^{t_m} (t_m-s)^{\alpha-1} s^{\alpha-1} (s-t_j) \rd s 
\leq C t_{m}^{2\alpha-1} h.
\end{align}

According to \eqref{eq.err5term}, we need to bound $\{J_i\}_{i=1}^{5}$. Firstly, Lemma \ref{lem.Opera}(1) shows
\begin{align} \label{eq.J1}
J_{1}
&\leq C \sum_{j=0}^{m-1} \int_{t_j}^{t_{j+1}} \| \cS_{0}(t_m-s) \|_{\cL(H)} \| X(t_j) - Y_j \|_{L^2(\Omega,H)} \rd s \nonumber\\
&\leq C \sum_{j=0}^{m-1} \int_{t_j}^{t_{j+1}} (t_m-s)^{\alpha-1} \| X(t_j) - Y_j \|_{L^2(\Omega,H)} \rd s.
\end{align}
Secondly, Lemma \ref{lem.Opera}, $\| X(0) \|_{L^2(\Omega,\dot{H}^{2\beta})} < \infty$ and \eqref{eq.usetoJ2} indicate
\begin{align} \label{eq.J2}
J_{2}
&\leq C \sum_{j=0}^{m-1} \int_{t_j}^{t_{j+1}} \| \cS_{0}(t_m-s) \|_{\cL(H)} \| A^{-\beta} ( \cS_{1-\alpha}(s) - \cS_{1-\alpha}(t_j) ) \|_{\cL(H)} \| X_0 \|_{L^2(\Omega,\dot{H}^{2\beta})} \rd s \nonumber\\
&\leq C \sum_{j=0}^{m-1} \int_{t_j}^{t_{j+1}} (t_m-s)^{\alpha-1} \big( s^{\alpha} - t_j^{\alpha} \big) \rd s \nonumber\\
&\leq C t_{m}^{2\alpha-1} h.
\end{align}

As for $J_3$, it holds that
\begin{align*}
J_{3}
&= \Big\| \int_{0}^{t_{1}} \cS_{0}(t_m-s) F'(X(t_0)) \Upsilon_{F\circ X}(s) \rd s \Big\|_{L^2(\Omega,H)} \\
&\quad + \Big\| \sum_{j=1}^{m-1} \int_{t_j}^{t_{j+1}} \cS_{0}(t_m-s) F'(X(t_j)) \big( \Upsilon_{F\circ X}(s) - \Upsilon_{F\circ X}(t_j) \big) \rd s \Big\|_{L^2(\Omega,H)} \\
&=: J_3^{\star} + J_3^{\star\star}.
\end{align*}
Here, by Lemma \ref{lem.Opera}(1) and Theorem \ref{prop.Regu},
\begin{align*}
J_3^{\star}
&\leq C \int_{0}^{t_{1}} \| \cS_{0}(t_m-s) \|_{\cL(H)} \int_{0}^{s} \| \cS_{0}(s-u) \|_{\cL(H)} \| X(u) \|_{L^2(\Omega,H)} \rd u \rd s \\
&\leq C \int_{0}^{t_{1}} (t_m-s)^{\alpha-1} \int_{0}^{s} (s-u)^{\alpha-1} \rd u \rd s \\
&\leq C t_m^{2\alpha-1} h.
\end{align*}
With Proposition \ref{th.sinRegu} at hand, applying Proposition \ref{prop.UpsilonG} with $\mu = \frac{1}{2}-\varepsilon$, $\nu = \min\{ \frac{\alpha r}{2\beta} + (\gamma-\frac{1}{2})^{+}, 1-\varepsilon \}$ and $G = F\circ X$ shows that for any $s \in [t_j,t_{j+1})$ with $j = 1,\cdots,m-1$, 
\begin{align*} 
\| \Upsilon_{F\circ X}(s) - \Upsilon_{F\circ X}(t_j) \|_{L^2(\Omega,H)}
&\leq C ( s^{\alpha} - t_j^{\alpha} ) + C t_j^{\varepsilon-\frac{1}{2}} h^{\min\{ \alpha + \frac{\alpha r}{2\beta} + (\gamma-\frac{1}{2})^{+} - \varepsilon, 1 \}},
\end{align*}
which together with \eqref{eq.usetoJ2}, \eqref{esti.tj->s} and the Beta function implies 
\begin{align*}
J_3^{\star\star}
&\leq C \sum_{j=1}^{m-1} \int_{t_j}^{t_{j+1}} \| \cS_{0}(t_m-s) \|_{\cL(H)} \| \Upsilon_{F\circ X}(s) - \Upsilon_{F\circ X}(t_j) \|_{L^2(\Omega,H)} \rd s \\
&\leq C \sum_{j=1}^{m-1} \int_{t_j}^{t_{j+1}} (t_m-s)^{\alpha-1} ( s^{\alpha} - t_j^{\alpha} ) \rd s 
+ C h^{\min\{ \alpha + \frac{\alpha r}{2\beta} + (\gamma-\frac{1}{2})^{+} - \varepsilon, 1 \}} \sum_{j=1}^{m-1} \int_{t_j}^{t_{j+1}} (t_m-s)^{\alpha-1} t_j^{\varepsilon-\frac{1}{2}} \rd s \\
&\leq C t_{m}^{2\alpha-1} h 
+ C h^{\min\{ \alpha + \frac{\alpha r}{2\beta} + (\gamma-\frac{1}{2})^{+} - \varepsilon, 1 \}} \int_{t_1}^{t_{m}} (t_m-s)^{\alpha-1} s^{\varepsilon-\frac{1}{2}} \rd s \\
&\leq C t_m^{2\alpha-1} h^{ \min\{ \alpha + \frac{\alpha r}{2\beta} + (\gamma-\frac{1}{2})^{+} - \varepsilon, 1 \} }.
\end{align*}
Thus, we obtain
\begin{align} \label{eq.J3}
J_{3} \leq C t_m^{2\alpha-1} h^{ \min\{ \alpha + \frac{\alpha r}{2\beta} + (\gamma-\frac{1}{2})^{+} - \varepsilon, 1 \} }.
\end{align}

As for $J_4$, it holds that
\begin{align} \label{eq.J4split2}
J_{4} &\leq \Big\| \sum_{j=0}^{m-1} \int_{t_j}^{t_{j+1}} \cS_{0}(t_m-s) F'(X(t_j)) \int_{t_j}^{s} \cS_{\gamma}(s-u) \rd W(u) \rd s \Big\|_{L^2(\Omega,H)} \nonumber\\
&\quad + \Big\| \sum_{j=0}^{m-1} \int_{t_j}^{t_{j+1}} \cS_{0}(t_m-s) F'(X(t_j)) \int_0^{t_j} \cS_{\gamma}(s-u) - \cS_{\gamma}(t_j-u) \rd W(u) \rd s \Big\|_{L^2(\Omega,H)} \nonumber\\
&=: J_4^{\star} + J_4^{\star\star}.
\end{align}
Here, it follows from the linearity of $F$, the independence of the noise increments, Lemma \ref{lem.Opera}(1), \eqref{eq.Lambda1sub} and Cauchy--Schwarz's inequality that
\begin{align*}
|J_4^{\star}|^2
&\leq C \sum_{j=0}^{m-1} \Big\| \int_{t_j}^{t_{j+1}} \cS_{0}(t_m-s) \int_{t_j}^{s} \cS_{\gamma}(s-u) \rd W(u) \rd s \Big\|_{L^2(\Omega,H)}^2 \nonumber\\
&\leq C \sum_{j=0}^{m-1} \Big( \int_{t_j}^{t_{j+1}} \| A^{\frac{\kappa- r}{2}} \cS_{0}(t_m-s) \|_{\cL(H)} \Big\| \int_{t_j}^{s} \cS_{\gamma}(s-u) \rd W(u) \Big\|_{L^2(\Omega,\dot{H}^{ r-\kappa})} \rd s \Big)^2 \nonumber\\
&\leq C \sum_{j=0}^{m-1} \Big( \int_{t_j}^{t_{j+1}} (t_m-s)^{ \alpha - 1 -\frac{\alpha}{2\beta}(\kappa- r) } (s-t_j)^{ \alpha + \gamma - \frac{1}{2} } \rd s \Big)^2 \nonumber\\
&\leq C h^{2(\alpha+\gamma-\frac{1}{2})} \sum_{j=0}^{m-1} \int_{t_j}^{t_{j+1}} (t_m-s)^{2\varepsilon-1} \rd s \int_{t_j}^{t_{j+1}} (t_m-s)^{ 2\alpha - 1 -\frac{\alpha}{\beta}(\kappa- r) - 2\varepsilon } \rd s \nonumber\\
&\leq C h^{ 2\min\{\alpha + \frac{\alpha r}{2\beta} + (\gamma-\frac{1}{2})^{+} - \varepsilon, \alpha+\gamma\} }.
\end{align*}
Similarly, by the linearity of $F$, $J_4^{\star\star}$ can be further split into

\begin{align} \label{eq.J42star2}
|J_4^{\star\star}|^2 &\leq C \sum_{j=0}^{m-1} \Big\| \int_{t_j}^{t_{j+1}} \cS_{0}(t_m-s) \int_0^{t_j} \cS_{\gamma}(s-u) - \cS_{\gamma}(t_j-u) \rd W(u) \rd s \Big\|_{L^2(\Omega,H)}^2 \nonumber\\
&\quad + C \sum_{0 \leq i < j \leq m-1} \hE \Big\< \int_{t_i}^{t_{i+1}} \cS_{0}(t_m-s) \int_0^{t_i} \cS_{\gamma}(s-u) - \cS_{\gamma}(t_i-u) \rd W(u) \rd s, \nonumber\\
&\qquad\qquad\qquad\quad\ \ \! \int_{t_j}^{t_{j+1}} \cS_{0}(t_m-\tau) \int_0^{t_j} \cS_{\gamma}(\tau-v) - \cS_{\gamma}(t_j-v) \rd W(v) \rd \tau \Big\> \nonumber\\
&=: C J_{4,1}^{\star\star} + C J_{4,2}^{\star\star}.
\end{align}
Then, using Lemma \ref{lem.Opera}, \eqref{eq.Lambda2sub} and Cauchy--Schwarz's inequality shows
\begin{align*}
J_{4,1}^{\star\star}
&\leq \sum_{j=0}^{m-1} \Big( \int_{t_j}^{t_{j+1}} \| A^{\frac{\kappa- r}{2}} \cS_{0}(t_m-s) \|_{\cL(H)} \Big\| \int_0^{t_j} \cS_{\gamma}(s-u) - \cS_{\gamma}(t_j-u) \rd W(u) \Big\|_{L^2(\Omega,\dot{H}^{ r-\kappa})} \rd s \Big)^2 \\
&\leq C \sum_{j=0}^{m-1} \Big( \int_{t_j}^{t_{j+1}} (t_m-s)^{ \alpha - 1 -\frac{\alpha}{2\beta}(\kappa- r) } (s-t_j)^{ \min\{ \alpha + \gamma - \frac{1}{2}, 1-\varepsilon \} } \rd s \Big)^2 \\
&\leq C h^{ 2\min\{ \alpha + \gamma - \frac{1}{2}, 1-\varepsilon \} } \sum_{j=0}^{m-1} \int_{t_j}^{t_{j+1}} (t_m-s)^{2\varepsilon-1} \rd s \int_{t_j}^{t_{j+1}} (t_m-s)^{ 2\alpha - 1 -\frac{\alpha}{\beta}(\kappa- r) - 2\varepsilon } \rd s \\
&\leq C h^{ 2\min\{\alpha + \frac{\alpha r}{2\beta} + (\gamma-\frac{1}{2})^{+} - \varepsilon, \alpha+\gamma, 1 \} }.
\end{align*}
The estimate of $J_{4,2}^{\star\star}$ is deferred to Lemma \ref{lem.J4star2}. With Lemma \ref{lem.J4star2} in mind, collecting the above estimates about $J_4$ implies
\begin{align} \label{eq.J4}
J_4 \leq C h^{ \min\{ \alpha + \frac{\alpha r}{2\beta} + (\gamma-\frac{1}{2})^{+} - \varepsilon, 1, \alpha+\gamma -\varepsilon \} }.
\end{align}
Therefore, combining the estimates \eqref{eq.err5term}, \eqref{eq.J1}--\eqref{eq.J3}, \eqref{eq.J4} and the fact $J_{5} = 0$ as well as applying the singular-type Gr\"onwall's inequality 
completes the proof of Theorem \ref{th.MLEulerErrorLinear}.
\hfill$\Box$

In order to bound $J_{4,2}^{\star\star}$, we prepare the following lemma. 
\begin{Lem} \label{lem.usetoJ35sum}
Let $a \in \hR$, $b < 1$, $c < 1$ and the integer $M \geq 2$. For $j \in \{0,1,\cdots, M\}$, put $t_j = j\frac{T}{M}$. Then there exists some positive constant $C=C(a,b,c,T)$ such that
\begin{align*}
\sup_{2 \leq m \leq M} \sum_{0 \leq i < j \leq m-1} \int_{t_i}^{t_{i+1}} \int_{t_j}^{t_{j+1}} (t_m-s)^{-a} (t_m-\tau)^{-b} (\tau-s)^{-c} \rd \tau \rd s 
\leq
\begin{cases}
C, & \mbox{if } a+b+c < 2, \\
C \ln M, & \mbox{if } a+b+c = 2, \\
C M^{a+b+c-2} , & \mbox{if } a+b+c > 2.
\end{cases}
\end{align*}
\end{Lem}

\begin{proof}
Using the change of variables and the Beta function shows
\begin{align*}
&\sum_{0 \leq i < j \leq m-1} \int_{t_i}^{t_{i+1}} \int_{t_j}^{t_{j+1}} (t_m-s)^{-a} (t_m-\tau)^{-b} (\tau-s)^{-c} \rd \tau \rd s \\
&\qquad\quad \leq \int_{0}^{t_{m-1}} \int_{s}^{t_m} (t_m-s)^{-a} (t_m-\tau)^{-b} (\tau-s)^{-c} \rd \tau \rd s \\
&\qquad\quad = \int_{t_1}^{t_m} \int_{0}^{u} u^{-a} v^{-b} (u-v)^{-c} \rd v \rd u \\
&\qquad\quad = B(1-b,1-c) \int_{t_1}^{t_m} u^{1-(a+b+c)} \rd u.
\end{align*}
Then, the proof can be completed by the direct calculations. \qed
\end{proof}

\begin{Lem} \label{lem.J4star2}
Let $J_{4,2}^{\star\star}$ be defined by \eqref{eq.J42star2}. Then there exists some positive constant $C$ such that
\begin{align*}
J_{4,2}^{\star\star} \leq C h^{ 2\min\{ \alpha + \frac{\alpha r}{2\beta} + (\gamma-\frac{1}{2})^{+} - \varepsilon, 1, \alpha+\gamma -\varepsilon \} }.
\end{align*}
\end{Lem}

\begin{proof}
To facilitate the proof, for $s \in (t_i,t_{i+1}]$, $\tau \in (t_j,t_{j+1}]$ with $0 \leq i < j \leq m-1$, introduce
\begin{align*}
K_{i,j}(s,\tau)
&:= \hE \Big\< \cS_{0}(t_m-s) \int_0^{t_i} \cS_{\gamma}(s-u) - \cS_{\gamma}(t_i-u) \rd W(u), \\
&\qquad\ \ \! \cS_{0}(t_m-\tau) \int_0^{t_j} \cS_{\gamma}(\tau-v) - \cS_{\gamma}(t_j-v) \rd W(v) \Big\>, \\
K_{i,j}^{\star}(s,\tau)
&:= \int_{0}^{t_i} \| \cS_{\gamma}(s-u) - \cS_{\gamma}(t_i-u) \|_{\cL(H)} \| \cS_{\gamma}(\tau-u) - \cS_{\gamma}(t_j-u) \|_{\cL(H)} \rd u.
\end{align*}
In fact, one can derive from \cite[Corollary 4.29]{DaPrato2014Book} that
\begin{align*}
K_{i,j}(s,\tau)
&= \int_{0}^{t_i} \textup{Tr}\Big( \cS_{0}(t_m-\tau) \big( \cS_{\gamma}(\tau-u) - \cS_{\gamma}(t_j-u) \big) Q \big( \cS_{\gamma}(s-u) - \cS_{\gamma}(t_i-u) \big) \cS_{0}(t_m-s) \Big) \rd u \\
&\leq \int_{0}^{t_i} \| Q^{\frac{1}{2}} \big( \cS_{\gamma}(s-u) - \cS_{\gamma}(t_i-u) \big) \cS_{0}(t_m-s) \|_{\cL_2(H)} \\
&\qquad\quad \times \| Q^{\frac{1}{2}} \big( \cS_{\gamma}(\tau-u) - \cS_{\gamma}(t_j-u) \big) \cS_{0}(t_m-\tau) \|_{\cL_2(H)} \rd u \\
&\leq \| A^{\frac{ r-\kappa}{2}} Q^{\frac{1}{2}} \|_{\cL_2(H)}^2 \| A^{\frac{\kappa- r}{2}} \cS_{0}(t_m-s) \|_{\cL(H)} \| A^{\frac{\kappa- r}{2}} \cS_{0}(t_m-\tau) \|_{\cL(H)} \times K_{i,j}^{\star}(s,\tau),
\end{align*}
which together with Assumption \ref{ass.Noise} and Lemma \ref{lem.Opera}(1) indicates
\begin{align} \label{esti.preJ35star}
J_{4,2}^{\star\star}
& = \sum_{0 \leq i < j \leq m-1} \int_{t_i}^{t_{i+1}} \int_{t_j}^{t_{j+1}} K_{i,j}(s,\tau) \rd \tau \rd s \nonumber\\
&\leq C\sum_{0 \leq i < j \leq m-1} \int_{t_i}^{t_{i+1}} \int_{t_j}^{t_{j+1}} (t_m-s)^{\alpha-1-\frac{\alpha}{2\beta}(\kappa-r)} (t_m-\tau)^{\alpha-1-\frac{\alpha}{2\beta}(\kappa- r)} K_{i,j}^{\star}(s,\tau) \rd \tau \rd s.
\end{align}
Here, it follows from Lemma \ref{lem.Opera}(2)--(3) that for $s \in (t_i,t_{i+1}]$, $\tau \in (t_j,t_{j+1}]$ with $0 \leq i < j \leq m-1$,
\begin{align*}
\ K_{i,j}^{\star}(s,\tau) 
\leq
\begin{cases}
C \int_{0}^{t_i} \big( (t_i-u)^{\alpha+\gamma-1} - (s-u)^{\alpha+\gamma-1} \big) \big( (t_j-u)^{\alpha+\gamma-1} - (\tau-u)^{\alpha+\gamma-1} \big) \rd u, & \mbox{if } \alpha + \gamma \in (\frac{1}{2},1), \\
C \int_{0}^{t_i} \int_{t_i-u}^{s-u} v^{-1} \rd v \int_{t_j-u}^{\tau-u} \omega^{-1} \rd \omega \rd u, & \mbox{if } \alpha + \gamma =1, \\
C \int_{0}^{t_i} \big( (s-u)^{\alpha+\gamma-1} - (t_i-u)^{\alpha+\gamma-1} \big) \big( (\tau-u)^{\alpha+\gamma-1} - (t_j-u)^{\alpha+\gamma-1} \big) \rd u, & \mbox{if } \alpha + \gamma \in (1,2].
\end{cases}
\end{align*}
Next, we put $s \in (t_i,t_{i+1}]$, $\tau \in (t_j,t_{j+1}]$ with $0 \leq i < j \leq m-1$ and estimate $J_{4,2}^{\star\star}$ in three cases.

\underline{Case 1:\ $\alpha + \gamma \in (\frac{1}{2},1)$}. By the change of variables, it holds that
\begin{align*}
&\quad\ \int_{0}^{t_i} (t_i-u)^{\alpha+\gamma-1} \big( (t_j-u)^{\alpha+\gamma-1} - (\tau-u)^{\alpha+\gamma-1} \big) \rd u \\
&= t_i^{\alpha+\gamma} \int_{0}^{1} (1-v)^{\alpha+\gamma-1} \big( (t_j-t_i v)^{\alpha+\gamma-1} - (\tau-t_i v)^{\alpha+\gamma-1} \big) \rd v \\
&\leq s^{\alpha+\gamma} \int_{0}^{1} (1-v)^{\alpha+\gamma-1} \big( (t_j-s v)^{\alpha+\gamma-1} - (\tau-s v)^{\alpha+\gamma-1} \big) \rd v \\
&= \int_{0}^{s} (s-u)^{\alpha+\gamma-1} \big( (t_j-u)^{\alpha+\gamma-1} - (\tau-u)^{\alpha+\gamma-1} \big) \rd u.
\end{align*}
Then,
\begin{align*}
K_{i,j}^{\star}(s,\tau)
&\leq C \bigg( \int_{0}^{t_i} (t_i-u)^{\alpha+\gamma-1} \big( (t_j-u)^{\alpha+\gamma-1} - (\tau-u)^{\alpha+\gamma-1} \big) \rd u \\
&\qquad\ - \int_{0}^{s} (s-u)^{\alpha+\gamma-1} \big( (t_j-u)^{\alpha+\gamma-1} - (\tau-u)^{\alpha+\gamma-1} \big) \rd u \\
&\qquad\ + \int_{t_i}^{s} (s-u)^{\alpha+\gamma-1} \big( (t_j-u)^{\alpha+\gamma-1} - (\tau-u)^{\alpha+\gamma-1} \big) \rd u \bigg) \\
&\leq C \int_{t_i}^{s} (s-u)^{\alpha+\gamma-1} \big( (t_j-u)^{\alpha+\gamma-1} - (\tau-u)^{\alpha+\gamma-1} \big) \rd u.
\end{align*}
When $j > i+1$, using the mean value theorem and $\tau-s \leq 2(t_j - u)$ with $u \in [t_i,s]$ shows
\begin{align*}
K_{i,j}^{\star}(s,\tau)
&\leq C \int_{t_i}^{s} (s-u)^{\alpha+\gamma-1} (\tau-t_j) (t_j - u)^{\alpha+\gamma-2} (t_j - u)^{1-\varepsilon} \rd u \, (\tau-s)^{\varepsilon-1} \\
&\leq C h (\tau-s)^{\varepsilon-1} \int_{t_i}^{s} (s-u)^{2(\alpha+\gamma-1)-\varepsilon} \rd u \\
&\leq C h^{2(\alpha+\gamma)-\varepsilon} (\tau-s)^{\varepsilon-1}.
\end{align*}
When $j = i+1$, one has $\tau-s \leq 2 h$, which indicates
\begin{align*}
K_{i,j}^{\star}(s,\tau)
&\leq C \int_{t_i}^{s} (s-u)^{\alpha+\gamma-1} (t_j-u)^{\alpha+\gamma-1} \rd u \, h^{1-\varepsilon} (\tau-s)^{\varepsilon-1} \\
&\leq C h^{1-\varepsilon} (\tau-s)^{\varepsilon-1} \int_{t_i}^{s} (s-u)^{2(\alpha+\gamma-1)} \rd u \\
&\leq C h^{2(\alpha+\gamma)-\varepsilon} (\tau-s)^{\varepsilon-1}.
\end{align*}
Thus, $K_{i,j}^{\star}(s,\tau) \leq C h^{2(\alpha+\gamma)-\varepsilon} (\tau-s)^{\varepsilon-1}$. Then, by recalling \eqref{esti.preJ35star} and using Lemma \ref{lem.usetoJ35sum},
\begin{align*}
J_{4,2}^{\star\star}
\leq C h^{2\min\{\alpha+\gamma, \alpha + \frac{\alpha r}{2\beta} + (\gamma-\frac{1}{2})^{+}\} - \varepsilon}.
\end{align*}

\underline{Case 2:\ $\alpha + \gamma = 1$}. Using the change of variables gives
\begin{align*}
K_{i,j}^{\star}(s,\tau)
&\leq C \int_{0}^{t_i} \int_{t_i}^{s} (v-u)^{\varepsilon-1} (v-u)^{-\varepsilon} \rd v \int_{t_j}^{\tau} (\omega-u)^{\varepsilon-1} (\omega-u)^{-\varepsilon} \rd \omega \rd u \\
&\leq C \int_{0}^{t_i} (t_i-u)^{\varepsilon-1} (t_j-u)^{\varepsilon-1} \int_{t_i}^{s} (v-u)^{-\varepsilon} \rd v \int_{t_j}^{\tau} (\omega-u)^{-\varepsilon} \rd \omega \rd u \\
&\leq C h^{2-2\varepsilon} \int_{0}^{t_i} (t_i-u)^{\varepsilon-1} (t_j-u)^{\varepsilon-1} \rd u.
\end{align*}
When $j > i+1$, it follows from $\tau-s \leq 2 (t_j - u)$ with $u\in[0,t_i]$ that
\begin{align*}
K_{i,j}^{\star}(s,\tau)
\leq C h^{2-2\varepsilon} \int_{0}^{t_i} (t_i-u)^{\varepsilon-1} (t_j-u)^{\varepsilon-1} (t_j-u)^{1-\varepsilon} \rd u \, (\tau-s)^{\varepsilon-1} 
\leq C h^{2-2\varepsilon} (\tau-s)^{\varepsilon-1}.
\end{align*}
When $j = i+1$, it follows from $t_j-u \geq h$ with $u \in [0,t_i]$ and $\tau-s \leq 2 h$ that
\begin{align*}
K_{i,j}^{\star}(s,\tau)
\leq C h^{1-\varepsilon} \int_{0}^{t_i} (t_i-u)^{\varepsilon-1} \rd u \, (\tau-s)^{1-\varepsilon} (\tau-s)^{\varepsilon-1} 
\leq C h^{2-2\varepsilon} (\tau-s)^{\varepsilon-1}.
\end{align*}
Thus, $K_{i,j}^{\star}(s,\tau) \leq C h^{2-2\varepsilon} (\tau-s)^{\varepsilon-1}$. Then, by recalling \eqref{esti.preJ35star} and using Lemma \ref{lem.usetoJ35sum},
\begin{align*}
J_{4,2}^{\star\star}
\leq C h^{ 2\min\{1, \alpha + \frac{\alpha r}{2\beta} + (\gamma-\frac{1}{2})^{+} \} - 2\varepsilon }.
\end{align*}

\underline{Case 3:\ $\alpha + \gamma \in (1,2]$}. Applying the mean value theorem shows
\begin{align*}
K_{i,j}^{\star}(s,\tau)
&\leq C \int_{0}^{t_i} (t_i-u)^{\alpha+\gamma-2} (s-t_i) (t_j-u)^{\alpha+\gamma-2} (\tau-t_j) \rd u \\
&\leq C h^2 \int_{0}^{t_i} (t_i-u)^{\alpha+\gamma-2} (t_j-u)^{\alpha+\gamma-2} \rd u.
\end{align*}
When $\alpha + \gamma \in (\frac{3}{2},2]$, one gets
\begin{align*}
K_{i,j}^{\star}(s,\tau)
\leq C h^2 \int_{0}^{t_i} (t_i-u)^{2(\alpha+\gamma-2)} \rd u
\leq C h^2.
\end{align*}
When $\alpha + \gamma \in (1,\frac{3}{2}]$ and $j > i+1$, it follows from $\tau-s \leq 2 (t_j - u)$ with $u \in [0,t_i]$ that
\begin{align*}
K_{i,j}^{\star}(s,\tau)
&\leq C h^2 \int_{0}^{t_i} (t_i-u)^{\alpha+\gamma-2} (t_j-u)^{\alpha+\gamma-2} (t_j-u)^{-2(\alpha+\gamma-\frac{3}{2})+\varepsilon} \rd u \, (\tau-s)^{2(\alpha+\gamma-\frac{3}{2})-\varepsilon} \\
&\leq C h^2 (\tau-s)^{2(\alpha+\gamma-\frac{3}{2})-\varepsilon} \int_{0}^{t_i} (t_i-u)^{\varepsilon-1} \rd u \\
&\leq C h^2 (\tau-s)^{2(\alpha+\gamma-\frac{3}{2})-\varepsilon}.
\end{align*}
When $\alpha + \gamma \in (1,\frac{3}{2}]$ and $j = i+1$, it follows from $t_j-u \geq h$ with $u \in [0,t_i]$ and $\tau-s \leq 2 h$ that
\begin{align*}
K_{i,j}^{\star}(s,\tau)
&\leq C h^2 \int_{0}^{t_i} (t_i-u)^{\alpha+\gamma-2} (t_j-u)^{2(\alpha+\gamma)-3-\varepsilon} (t_j-u)^{1-(\alpha+\gamma)+\varepsilon} \rd u \\
&\leq C h^{2(\alpha+\gamma)-1-\varepsilon} \int_{0}^{t_i} (t_i-u)^{\varepsilon-1} \rd u \, (\tau-s)^{-2(\alpha+\gamma-\frac{3}{2})+\varepsilon} (\tau-s)^{2(\alpha+\gamma-\frac{3}{2})-\varepsilon} \\
&\leq C h^2 (\tau-s)^{2(\alpha+\gamma-\frac{3}{2})-\varepsilon}.
\end{align*}
Thus, $K_{i,j}^{\star}(s,\tau) \leq C h^2 (\tau-s)^{ \min\{0, 2(\alpha+\gamma-\frac{3}{2})-\varepsilon\} }$. Then, by recalling \eqref{esti.preJ35star} and using Lemma \ref{lem.usetoJ35sum},
\begin{align*}
J_{4,2}^{\star\star}
\leq C h^{ 2\min\{1, \alpha + \frac{\alpha r}{2\beta} + (\gamma-\frac{1}{2})^{+} - \varepsilon\} }.
\end{align*}
Hence, the proof is completed. \qed 
\end{proof}

\emph{Proof of Theorem \ref{th.MLEulerErrorNonlinear}}. Under Assumption \ref{ass.Lip1}, the estimates \eqref{eq.J1}--\eqref{eq.J3} also hold. Thus, it remains to re-estimate $J_4$ and $J_5$. As shown in \eqref{eq.J4split2}, $J_4$ has been broken down into $J_4^{\star}$ and $J_4^{\star\star}$. Using the independence of the noise increments, Assumption \ref{ass.Lip1}, Lemma \ref{lem.Opera}(1) and \eqref{eq.Lambda1sub}, Cauchy--Schwarz's inequality and the identity $\alpha + \gamma - \frac{1}{2} - \frac{\alpha\kappa}{2\beta} = (\gamma-\frac{1}{2})^{+} + \frac{\alpha\varepsilon_0}{2\beta}$ reveals
\begin{align} \label{eq.J4StarNon}
|J_4^{\star}|^2
&= \sum_{j=0}^{m-1} \Big\| \int_{t_j}^{t_{j+1}} \cS_{0}(t_m-s) F'(X(t_j)) \int_{t_j}^{s} \cS_{\gamma}(s-u) \rd W(u) \rd s \Big\|_{L^2(\Omega,H)}^2 \nonumber\\
&\leq C \sum_{j=0}^{m-1} \Big( \int_{t_j}^{t_{j+1}} \| \cS_{0}(t_m-s) \|_{\cL(H)} \Big\| \int_{t_j}^{s} \cS_{\gamma}(s-u) \rd W(u) \Big\|_{L^2(\Omega,H)} \rd s \Big)^2 \nonumber\\
&\leq C \sum_{j=0}^{m-1} \Big( \int_{t_j}^{t_{j+1}} (t_m-s)^{\alpha-1} (s-t_j)^{\alpha + \gamma - \frac{1}{2} - \frac{\alpha}{2\beta} (\kappa- r) } \rd s \Big)^2 \nonumber\\
&\leq C h^{2( \frac{\alpha}{2\beta}(r+\varepsilon_0) + (\gamma-\frac{1}{2})^{+} )} \sum_{j=0}^{m-1} \int_{t_j}^{t_{j+1}} (t_m-s)^{2\varepsilon-1} \rd s \int_{t_j}^{t_{j+1}} (t_m-s)^{2\alpha-1-2\varepsilon} \rd s \nonumber\\
&\leq C h^{ 2( \min\{ \alpha-\varepsilon, \frac{1}{2}\} + \frac{\alpha}{2\beta}(r+\varepsilon_0) + (\gamma-\frac{1}{2})^{+} ) }. 
\end{align}
According to Lemma \ref{lem.Opera}(1), \eqref{ass.1stF}, H\"older's inequality, Theorem \ref{prop.Regu} and \eqref{eq.Lambda2sub}, one has
\begin{align*}
J_4^{\star\star}
&\leq \sum_{j=0}^{m-1} \int_{t_j}^{t_{j+1}} \| A^{\frac{\delta}{2}} \cS_{0}(t_m-s) \|_{\cL(H)} \Big\| A^{-\frac{\delta}{2}} F'(X(t_j)) \int_0^{t_j} \cS_{\gamma}(s-u) - \cS_{\gamma}(t_j-u) \rd W(u) \Big\|_{L^2(\Omega,H)} \rd s \\
&\leq C \sum_{j=0}^{m-1} \int_{t_j}^{t_{j+1}} (t_m-s)^{\alpha-1-\frac{\alpha\delta}{2\beta}} \Big( 1 + \| X(t_j) \|_{L^4(\Omega,\dot{H}^{r})} \Big)  \Big\| \int_0^{t_j} \cS_{\gamma}(s-u) - \cS_{\gamma}(t_j-u) \rd W(u) \Big\|_{L^4(\Omega,\dot{H}^{-r})} \rd s \\
&\leq
\begin{cases}
C h^{\frac{1}{2} + \frac{\alpha}{2\beta}(2r-\kappa)}, & \mbox{if } \alpha + \gamma = 1, \\
C h^{\min\{\frac{\alpha}{2\beta}\min\{\kappa,2r\} + (\gamma-\frac{1}{2})^{+} + \frac{\alpha\varepsilon_0}{2\beta}, 1-\varepsilon\}}, & \mbox{if } \alpha + \gamma \neq 1.
\end{cases}
\end{align*}
Thus, it follows from \eqref{eq.J4split2} and \eqref{eq.J4StarNon} that
\begin{align} \label{eq.J4Non}
J_4 \leq
\begin{cases}
C h^{\frac{1}{2} + \frac{\alpha}{2\beta}(2r-\kappa)}, & \mbox{if } \alpha + \gamma = 1, \\
C h^{\min\{\frac{\alpha}{2\beta}\min\{\kappa,2r\} + (\gamma-\frac{1}{2})^{+} + \frac{\alpha\varepsilon_0}{2\beta}, 1-\varepsilon\}}, & \mbox{if } \alpha + \gamma \neq 1,
\end{cases}
\end{align}
where we used the following inequality
\begin{align*}
\min\{ \alpha-\varepsilon, \frac{1}{2}\} + \frac{\alpha r}{2\beta} + (\gamma-\frac{1}{2})^{+} + \frac{\alpha\varepsilon_0}{2\beta}
\geq
\begin{cases}
\frac{1}{2} + \frac{\alpha}{2\beta}(2r-\kappa), & \mbox{if } \alpha + \gamma = 1, \\
\frac{\alpha}{2\beta}\min\{\kappa,2r\} + (\gamma-\frac{1}{2})^{+} + \frac{\alpha\varepsilon_0}{2\beta}, & \mbox{if } \alpha + \gamma \neq 1.
\end{cases}
\end{align*}
Here, the proof of the case $\alpha + \gamma = 1$ is not difficult, and it is used in the case $\alpha + \gamma \neq 1$ that
\begin{align*}
\frac{\alpha}{2\beta}\min\{\kappa-r, r\}
\leq \frac{\alpha}{2\beta} \frac{\kappa}{2}
\leq \frac{\alpha}{2} \leq \min\{\alpha-\varepsilon,\frac{1}{2}\}.
\end{align*}

For $J_5$, making use of the condition \eqref{ass.2ndF}, Lemma \ref{lem.Opera}(1) and Theorem \ref{prop.Regu}, we have
 \begin{align*}
 J_{5}
&\leq \sum_{j=0}^{m-1} \int_{t_j}^{t_{j+1}} \| A^{\frac{\zeta}{2}} \cS_{0}(t_m-s) \|_{\cL(H)} \| A^{-\frac{\zeta}{2}} R_{F,j}(s) \|_{L^2(\Omega,H)} \rd s \nonumber\\
&\leq C \sum_{j=0}^{m-1} \int_{t_j}^{t_{j+1}} (t_m-s)^{\alpha-1-\frac{\alpha\zeta}{2\beta}} \| X(s) - X(t_j) \|_{L^4(\Omega,H)}^2 \rd s.
\end{align*}
Further applying Theorem \ref{prop.Regu}, Proposition \ref{th.sinRegu}, \eqref{esti.tj->s} and the Beta function yields
\begin{align} \label{eq.J5Non}
J_{5}
&\leq C \int_{0}^{t_{1}} (t_m-s)^{\alpha-1-\frac{\alpha\zeta}{2\beta}} (s-t_j)^{2\min\{ \alpha, \frac{\alpha r}{2\beta} + (\gamma-\frac{1}{2})^{+} \}} \rd s \nonumber\\
&\quad + C \sum_{j=1}^{m-1} \int_{t_j}^{t_{j+1}} (t_m-s)^{\alpha-1-\frac{\alpha\zeta}{2\beta}} t_j^{2\varepsilon-1} (s-t_j)^{ 2\min\{ \frac{\alpha r}{2\beta} + (\gamma-\frac{1}{2})^{+}, 1-\varepsilon\} } \rd s \nonumber\\
&\leq C t_m^{\alpha-1-\frac{\alpha\zeta}{2\beta}} h + C h^{ 2\min\{ \frac{\alpha r}{2\beta} + (\gamma-\frac{1}{2})^{+}, 1-\varepsilon\} } \int_{t_1}^{t_m} (t_m-s)^{\alpha-1-\frac{\alpha\zeta}{2\beta}} s^{2\varepsilon-1} \rd s \nonumber\\
&\leq C t_m^{\alpha-1-\frac{\alpha\zeta}{2\beta}} h^{ 2\min\{ \frac{\alpha r}{2\beta} + (\gamma-\frac{1}{2})^{+}, \frac{1}{2} \} } \nonumber\\
&\leq
\begin{cases}
C t_m^{\alpha-1-\frac{\alpha\zeta}{2\beta}} h^{ \frac{\alpha r}{\beta} + \min\{ 2(\gamma-\frac{1}{2})^{+}, \frac{1}{2} - \frac{\alpha\kappa}{2\beta} \} }, & \mbox{if } \alpha + \gamma = 1, \\
C t_m^{\alpha-1-\frac{\alpha\zeta}{2\beta}} h^{\min\{\frac{\alpha}{2\beta}\min\{\kappa,2r\} + (\gamma-\frac{1}{2})^{+}, 1-\varepsilon\}}, & \mbox{if } \alpha + \gamma \neq 1.
\end{cases}
\end{align}
Here, for the case $\alpha + \gamma \neq 1$, the last step also used the inequality
\begin{align*}
\frac{\alpha}{2\beta}\min\{\kappa-r, r\}
\leq \frac{\alpha r}{2\beta} \leq \frac{\alpha r}{2\beta} + (\gamma-\frac{1}{2})^{+}.
\end{align*}
Finally, collecting the estimates \eqref{eq.err5term}, \eqref{eq.J1}--\eqref{eq.J3}, \eqref{eq.J4Non} and \eqref{eq.J5Non} as well as applying the singular-type Gr\"onwall's inequality completes the proof of Theorem \ref{th.MLEulerErrorNonlinear}.
\hfill$\Box$

\subsection{Freidlin--Wentzell type LDP}
\label{S32}

Applying the MLE integrator to \eqref{eq.model-small}, then the associated numerical solution $Y_m^\epsilon\approx X^\epsilon(t_m)$ can be formulated as 
\begin{align*} 
Y_m^\epsilon = \cS_{1-\alpha}(t_m) X_0 + \sum_{j=0}^{m-1} \int_{t_j}^{t_{j+1}} \cS_{0}(t_m-s) F(Y_j^\epsilon) \rd s + \sqrt{\epsilon}\int_0^{t_m} \cS_{\gamma}(t_m-s) \rd W(s), 
\end{align*}
where $m = 1,2,\cdots,M$ and $Y_0^\epsilon = X_0$. Further, one can also define the continuified numerical solution $Y^\epsilon=\{Y^\epsilon(t),t\in[0,T]\}$ by 
\begin{align*}
Y^\epsilon(t) = \cS_{1-\alpha}(t) X_0 + \int_{0}^{t} \cS_{0}(t-s) F(Y^\epsilon(\lfloor s\rfloor_h)) \rd s + \sqrt{\epsilon}\int_0^{t} \cS_{\gamma}(t-s) \rd W(s),
\end{align*}
where $\lfloor s\rfloor_h:=h\left[\frac{s}{h}\right]$ with $[\,\cdot\,]$ being the floor function. 
The following proposition states the Freidlin--Wentzell type LDP of the continuified numerical solution $\{Y^{\epsilon}\}_{\epsilon>0}$.

\begin{Prop}\label{th:LDPnum}
Let $X_0 \in L^p(\Omega,\dot{H}^{2\beta})$ for some $p \geq 2$ and Assumptions \ref{ass.Noise}--\ref{ass.Lip1} hold.
Then the continuified numerical solution $\{Y^{\epsilon}\}_{\epsilon>0}$ satisfies an LDP on $\mathcal C([0,T],H)$ as $\epsilon\to 0$ with the good rate function $I^M$ given by
\begin{align*}
I^M(x)=\inf_{\left\{v\in L^2([0,T],H_0),\,x=Z_M^v\right\}}\frac12 \int_0^T|v(s)|_0^2\rd s,\quad x\in\mathcal C([0,T],H),
\end{align*}
where $Z_M^v=\{Z_M^v(t),t\in[0,T]\}$ is the solution of the following skeleton equation 
\begin{align*}
Z_M^v(t) = \cS_{1-\alpha}(t) X_0 + \int_0^t \cS_{0}(t-s) F(Z_M^v(\lfloor s\rfloor_h)) \rd s + \int_0^t \cS_{\gamma}(t-s) v(s)\rd s, \qquad t \in [0,T].
\end{align*}
\end{Prop}
The proof of Proposition \ref{th:LDPnum} is similar to that of Theorem \ref{th:LDPexact}, so it is omitted. The readers are referred to \cite{CHJS21,CHJS22} for more results on the LDPs for numerical methods of stochastic differential equations. The following theorem reveals that the large deviation rate function $I^M$ of $\{Y^{\epsilon}\}_{\epsilon>0}$~$\Gamma$-converges to the large deviation rate function $I$ of $\{X^{\epsilon}\}_{\epsilon>0}$, as the partition parameter $M$ tends to infinity. We refer to \cite{DG93} for a detailed introduction of the $\Gamma$-convergence.
\begin{Theo}\label{th:LDPratenum}
Let $X_0 \in L^p(\Omega,\dot{H}^{2\beta})$ for some $p \geq 2$ and Assumptions \ref{ass.Noise}--\ref{ass.Lip1} hold. Then
the sequence $\{I^M\}_{M\in\hN_+}$ $\Gamma$-converges to $I$ on $\mathcal C([0,T],H)$ as $M\to \infty$, i.e., the following properties hold:\
\begin{itemize}
\item[(1)] For any $x\in \mathcal C([0,T],H)$ and any sequence $\{x_M\}_{M\in\hN_+}$ converging to $x$ in $\mathcal C([0,T],H)$,
\begin{align*}
\liminf_{M\to \infty}I^M(x_M)\ge I(x).
\end{align*}
\item[(2)] For any $x\in \mathcal C([0,T],H)$, there is a recovery sequence $\{x_M\}_{M\in\hN_+}$ converging to $x$ in $\mathcal C([0,T],H)$ such that 
\begin{align*}
\limsup_{M\to \infty}I^M(x_M)\le I(x).
\end{align*}
\end{itemize}
\end{Theo}

\begin{proof}
(1) Without loss of generality, we assume that $K_0:=\liminf_{M\to \infty}I^M(x_M)<\infty$; otherwise the proof is completed.
Let $\{x_M\}_{M\in \hN_+}$ be any sequence converging to $x$ in $\mathcal C([0,T],H)$. Then there exists a subsequence $\{M_k\}_{k\in\hN_+}$ satisfying 
\begin{align*}
\liminf_{M\to \infty}I^M(x_M)=\lim_{k\to\infty}I^{M_k}(x_{M_k})=K_0\in[0,\infty).
\end{align*}
Consequently, there exists some positive constant $K_1$ such that $I^{M_k}(x_{M_k})\le K_1$ for all $k\in\hN_+$. 
According to the definition of $I^M$, there exists a sequence $\{v_{M_k}\}_{k\in\hN_+}\subset L^2([0,T],H_0)$ such that $Z_{M_k}^{v_{M_k}}=x_{M_k}$ and 
\begin{align} \label{eq:vmk} 
\frac{1}{2}\int_0^T|v_{M_k}(s)|_0^2\rd s\le I^{M_k}(x_{M_k})+M_k^{-1},\qquad k\in\hN_+,
\end{align}
which leads to $\{v_{M_k}\}_{k\in\hN_+}\subseteq S_{2(K_1+1)}$. Since $S_{2(K_1+1)}$ endowed with the weak topology of $L^2([0,T],H_0)$ is a compact Polish space, for arbitrarily subsequence $\{g_k\}_{k\in \hN_+}\subseteq\{v_{M_k}\}_{k\in \hN_+}$, there exists a further subsequence of $\{g_k\}_{k\in \hN_+}$, still denoted by $\{g_k\}_{k\in \hN_+}$, such that $g_k\overset{w}\to g$ for some $g\in S_{2(K_1+1)}$ as $k\to \infty$. According to Lemma \ref{eq:weakcm}, we have that 
 $Z^{g_k}=X^{0,g_k}\to Z^g$ in $\mathcal C([0,T],H)$ as $k\to \infty$. 
 
By virtue of \eqref{Holder-Xepsilon} with $\tau_\epsilon\equiv 0$, we have that for all $v\in S_N$,
\begin{align*}
\|Z^v(t)-Z^v(s)\|\leq C (t-s)^{ \min\{ \alpha, \frac{\alpha r}{2\beta} + (\gamma-\frac{1}{2})^{+} \} }, \qquad \forall \, 0\leq s < t \leq T.
\end{align*}
Based on the singular Gronwall's inequality, one has that 
\begin{align}\label{eq:Zv}
\sup_{v\in S_N}\|Z_M^v -Z^v\|_{\mathcal C([0,T],H)}\le CM^{- \min\{ \alpha, \frac{\alpha r}{2\beta} + (\gamma-\frac{1}{2})^{+} \}}.
\end{align}
Thus $x=\lim_{k\to\infty}x_{M_k}=\lim_{k\to \infty}Z_{M_k}^{v_{M_k}}=\lim_{k\to \infty}Z^{v_{M_k}}$ in $\mathcal C([0,T],H)$. Taking into account that $\{g_k\}_{k\in \hN_+}$ is a subsequence of $\{v_{M_k}\}_{k\in \hN_+}$, we deduce $Z^g=\lim_{k\to \infty}Z^{g_k}=\lim_{k\to \infty}Z^{v_{M_k}}=x$ and thus by \eqref{eq:vmk},
\begin{align*}
I(x)\le \frac{1}{2}\int_0^T|g(s)|_0^2\rd s\le \frac{1}{2}\liminf_{k\to\infty}\int_0^T|g_k(s)|_0^2\rd s\le \liminf_{k\to\infty}I^{M_k}(x_{M_k})=\liminf_{M\to \infty}I^M(x_M). 
\end{align*}

(2) Without loss of generality, we assume that $I(x)<\infty;$ otherwise the conclusion holds naturally.
From the definition of $I$, there exists a sequence $\{v_M\}_{M\in\hN_+}\subseteq L^2([0,T],H_0)$ such that $Z^{v_M}=x$ and 
\begin{align*}
I(x)\ge \frac{1}{2}\int_0^T|v^M(s)|_0^2\rd s-M^{-1},\qquad\forall\, M\in\hN_+. 
\end{align*}
This also implies $\{v_M\}_{M\in\hN_+}\subseteq S_{2(I(x)+1)}$. Define $x_M:=Z_M^{v_M}$. Then it follows from \eqref{eq:Zv} that
\begin{align*}
\|x_M-x\|_{\mathcal C([0,T],H)}=\|Z_M^{v_M}-Z^{v_M}\|_{\mathcal C([0,T],H)}\le CM^{- \min\{ \alpha, \frac{\alpha r}{2\beta} + (\gamma-\frac{1}{2})^{+} \}}\to 0, \qquad \mbox{as } M\to\infty. 
\end{align*}
In addition, 
\begin{align*}
I^M(x_M)\le \frac{1}{2}\int_0^T|v^M(s)|_0^2\rd s\le I(x)+M^{-1},\qquad\forall\, M\in\hN_+, 
\end{align*}
which proves $\limsup_{M\to \infty}I^M(x_M)\le I(x)$, as required. \qed 
\end{proof}

\begin{appendices}

\section{Estimates of solution operator}
\label{sec.SolOpera}

We introduce the Mittag--Leffler function and some basic inequalities, which are taken from \cite[Chapter 1.2]{Podlubny1999}. Put $a \in (0,1)$ and $b \in \hR$. The Mittag--Leffler function is defined by 
\begin{align*}
E_{a, b} (z) = \sum_{k=0}^{\infty} \frac{z^k}{\Gamma(a k + b)}, \qquad \mbox{for } z \in \hC.
\end{align*}
For any real number $c \in (\frac{a\pi}{2}, a\pi)$, there exists some constant $C = C(a,b,c) > 0$ such that
\begin{align} \label{upper.ML1}
| E_{a, b} (z) | &\leq C (1 + |z|)^{-1}, \qquad c \leq |\arg(z)| \leq \pi.
\end{align}
When $b = a$, \eqref{upper.ML1} can be refined as
\begin{align} \label{upper.ML2}
| E_{a, a} (z) | &\leq C (1 + |z|)^{-2}, \qquad c \leq |\arg(z)| \leq \pi.
\end{align}
In addition, for any $\eta \in [0,1]$ and $\lambda > 0$, 
\begin{align} \label{eq.MLder}
\frac{\rd}{\rd t} [ t^{a+\eta-1} E_{a,a+\eta}(-\lambda t^{a}) ]
=
\begin{cases}
-\lambda t^{a-1} E_{a,a}(-\lambda t^{a}), & \mbox{if } a+\eta = 1, \\
t^{a+\eta-2} E_{a,a+\eta-1}(-\lambda t^{a}), & \mbox{if } a+\eta \neq 1.
\end{cases}
\end{align}

\begin{Lem} \label{lem.Opera}
Let $0 < s < t \leq T$, $\alpha \in (0,1)$, $\beta \in (0,1]$ and $\eta \in [0,1]$. Then there exists some positive constant $C$ independent of $t$ and $s$ such that $\cS_{\eta}(\cdot)$ defined by \eqref{eq.St} has the following estimates:\
\begin{enumerate}
\item[(1)] For any $\rho \leq 2\beta$,
\begin{align*}
\| A^{\frac{\rho}{2}} \cS_{\eta}(t) \|_{\cL(H)} \leq C t^{ \alpha + \eta - 1 - \frac{\alpha}{2\beta}\rho^{+} }.
\end{align*}

\item[(2)] When $\alpha + \eta = 1$, for any $\rho \in [-2\beta, 2\beta]$,
\begin{align*}
\| A^{\frac{\rho}{2}} ( \cS_{1-\alpha}(t) - \cS_{1-\alpha}(s) ) \|_{\cL(H)} \leq C \int_s^t u^{-1-\frac{\alpha\rho}{2\beta}} \rd u.
\end{align*}

\item[(3)] When $\alpha + \eta \neq 1$, for any $\rho \leq 2\beta$,
\begin{align*}
\| A^{\frac{\rho}{2}} ( \cS_{\eta}(t) - \cS_{\eta}(s) ) \|_{\cL(H)}
\leq C \int_s^t u^{\alpha+\eta-2-\frac{\alpha}{2\beta}\rho^{+}} \rd u.
\end{align*}
\end{enumerate}
\end{Lem}

\begin{proof}
(1) It can be obtained by \cite[Lemma 3.6]{Kang2021IMA} with trivial extensions, so the detail is omitted.

(2) By \eqref{eq.MLder} and \eqref{upper.ML2}, for any $\rho \in [-2\beta, 2\beta]$, one can derive
\begin{align*}
\| A^{\frac{\rho}{2}} ( \cS_{1-\alpha}(t) - \cS_{1-\alpha}(s) ) \|_{\cL(H)}
&= \sup_{k \geq 1} \Big\{ \lambda_k^{\frac{\rho}{2}} | E_{\alpha,1} (-\lambda_k^{\beta} t^{\alpha}) - E_{\alpha,1} (-\lambda_k^{\beta} s^{\alpha}) | \Big\} \\
&= \sup_{k \geq 1} \Big\{ \lambda_k^{\frac{\rho}{2}} \Big| \int_s^t -\lambda_k^{\beta} u^{\alpha-1} E_{\alpha,\alpha} (-\lambda_k^{\beta} u^{\alpha}) \rd u \Big| \Big\} \\
&\leq C \sup_{k \geq 1} \Big\{ \int_s^t u^{-1-\frac{\alpha\rho}{2\beta}} \frac{(\lambda_k^{\beta}u^{\alpha})^{1+\frac{\rho}{2\beta}}}{(1 + \lambda_k^{\beta}u^{\alpha})^2} \rd u \Big\} \\
&\leq C \int_s^t u^{-1-\frac{\alpha\rho}{2\beta}} \rd u.
\end{align*}

(3) When $\rho \in [0, 2\beta]$, it follows from \eqref{eq.MLder} and \eqref{upper.ML1} that
\begin{align*}
\| A^{\frac{\rho}{2}} ( \cS_{\eta}(t) - \cS_{\eta}(s) ) \|_{\cL(H)}
&= \sup_{k \geq 1} \Big\{ \lambda_k^{\frac{\rho}{2}} | t^{\alpha+\eta-1} E_{\alpha,\alpha+\eta} (-\lambda_k^{\beta} t^{\alpha}) - s^{\alpha+\eta-1} E_{\alpha,\alpha+\eta} (-\lambda_k^{\beta} s^{\alpha}) | \Big\} \\
&= \sup_{k \geq 1} \Big\{ \lambda_k^{\frac{\rho}{2}} \Big| \int_s^t u^{\alpha+\eta-2} E_{\alpha,\alpha+\eta-1} (-\lambda_k^{\beta} u^{\alpha}) \rd u \Big| \Big\} \\
&\leq C \sup_{k \geq 1} \Big\{ \int_s^t u^{\alpha+\eta-2-\frac{\alpha\rho}{2\beta}} \frac{ (\lambda_k^{\beta} u^{\alpha})^{\frac{\rho}{2\beta}} }{1 + \lambda_k^{\beta} u^{\alpha}} \rd u \Big\} \\
&\leq C \int_s^t u^{\alpha+\eta-2-\frac{\alpha\rho}{2\beta}} \rd u.
\end{align*}
While for any $\rho < 0$,
\begin{align*}
\| A^{\frac{\rho}{2}} ( \cS_{\eta}(t) - \cS_{\eta}(s) ) \|_{\cL(H)} \leq C \| \cS_{\eta}(t) - \cS_{\eta}(s) \|_{\cL(H)} \leq C \int_s^t u^{\alpha+\eta-2} \rd u.
\end{align*}
Hence, the proof is completed. \qed
\end{proof}

\begin{Lem} \label{le.DiffOpera}
Let $\alpha,\beta,\gamma,\kappa$ be given as in Assumption \ref{ass.Noise} and $\rho \in [-\kappa, \kappa]$. Then there exists some positive constant $C$ independent of $t$ and $s$ such that for any $0 < s < t \leq T$,
\begin{enumerate}
\item[(1)] When $\alpha + \gamma = 1$,
\begin{align*}
\int_0^s \| A^{\frac{\rho}{2}} ( \cS_{1-\alpha}(t-u) - \cS_{1-\alpha}(s-u) ) \|_{\cL(H)}^2 \rd u
\leq C (t-s)^{ 2(\frac{1}{2} - \frac{\alpha\rho}{2\beta}) }.
\end{align*}

\item[(2)] When $\alpha + \gamma \neq 1$ and $\rho = (\alpha+\gamma-1)\frac{2\beta}{\alpha}$,
\begin{align*}
\int_0^s \| A^{\frac{\rho}{2}} ( \cS_{\gamma}(t-u) - \cS_{\gamma}(s-u) ) \|_{\cL(H)}^2 \rd u
\leq C (t-s)^{2\min\{ \alpha+\gamma-\frac{1}{2}, \frac12 \}}.
\end{align*}

\item[(3)] When $\alpha + \gamma \neq 1$ and $\rho \neq (\alpha+\gamma-1)\frac{2\beta}{\alpha}$,
\begin{align*}
&\quad\ \int_0^s \| A^{\frac{\rho}{2}} ( \cS_{\gamma}(t-u) - \cS_{\gamma}(s-u) ) \|_{\cL(H)}^2 \rd u 
\leq C(t-s)^{2\min\{\alpha+\gamma-\frac{1}{2}-\frac{\alpha}{2\beta} \rho^{+}, 1-\varepsilon\}}.
\end{align*}
\end{enumerate}
\end{Lem}

\begin{proof}
To facilitate the proof, we firstly give the estimate
\begin{align} \label{inequ.minus1squre}
\int_0^s \Big| \int_{s}^{t} (v-u)^{-1} \rd v \Big|^2 \rd u
&\leq \int_0^s \int_{s}^{t} (v-u)^{-\frac{1}{2}} \rd v \int_{s}^{t} (v-u)^{-\frac{3}{2}} \rd v \rd u \nonumber\\
&\leq C(t-s)^{\frac{1}{2}} \int_0^s | (t-u)^{-\frac{1}{2}} - (s-u)^{-\frac{1}{2}} | \rd u \nonumber\\
&\leq C(t-s).
\end{align}

(1) Applying Lemma \ref{lem.Opera}(2) and the change of variables shows
\begin{align*}
\int_0^s \| A^{\frac{\rho}{2}} ( \cS_{1-\alpha}(t-u) - \cS_{1-\alpha}(s-u) ) \|_{\cL(H)}^2 \rd u
\leq C \int_0^s \Big| \int_{s-u}^{t-u} v^{-1-\frac{\alpha\rho}{2\beta}} \rd v \Big|^2 \rd u 
= C \int_0^s \Big| \int_{s}^{t} (v-u)^{-1-\frac{\alpha\rho}{2\beta}} \rd v \Big|^2 \rd u.
\end{align*}
Thus, when $\rho = 0$, \eqref{inequ.minus1squre} yields the desired result; while $\rho \in [-\kappa,0) \cup (0,\kappa]$, it follows from \cite[Lemma 3.2]{DaiXiaoBu2021} that
\begin{align*}
\int_0^s \| A^{\frac{\rho}{2}} ( \cS_{1-\alpha}(t-u) - \cS_{1-\alpha}(s-u) ) \|_{\cL(H)}^2 \rd u
\leq C \int_0^s \Big( (t-u)^{-\frac{\alpha\rho}{2\beta}} - (s-u)^{-\frac{\alpha\rho}{2\beta}} \Big)^2 \rd u  
\leq C (t-s)^{1 - \frac{\alpha\rho}{\beta}},
\end{align*}
where the last step also used the fact $-\frac{\alpha\rho}{2\beta} \in (-\frac{1}{2},\frac{1}{2})$. Indeed, it follows from $\alpha + \gamma = 1$ and $\kappa = \min\{ (\alpha+\gamma-\frac{1}{2})\frac{2\beta}{\alpha}, 2\beta \}-\varepsilon_0$ that $-\frac{\alpha\rho}{2\beta} \in (-\frac{1}{2},\frac{1}{2})$, since
\begin{itemize}
\item[$\bullet$] if $\alpha \in (0, \frac{1}{2}]$, then $\kappa = 2\beta-\varepsilon_0$ and $\rho \in (-2\beta, 2\beta)$;

\item[$\bullet$] if $\alpha \in (\frac{1}{2}, 1)$, then $\kappa = \frac{\beta}{\alpha}-\varepsilon_0$ and $\rho \in (-\frac{\beta}{\alpha}, \frac{\beta}{\alpha})$.
\end{itemize}

(2) When $\alpha + \gamma \in (\frac{1}{2}, 1)$, using Lemma \ref{lem.Opera}(3) and \cite[Lemma 3.2]{DaiXiaoBu2021} reads
\begin{align*}
\int_0^s \| A^{\frac{\rho}{2}} ( \cS_{\gamma}(t-u) - \cS_{\gamma}(s-u) ) \|_{\cL(H)}^2 \rd u
\leq C\int_0^s \big( (t-u)^{\alpha+\gamma-1} - (s-u)^{\alpha+\gamma-1} \big)^2 \rd u  
\leq C (t-s)^{2(\alpha+\gamma-\frac{1}{2})}.
\end{align*}
While $\alpha + \gamma \in (1, 2)$, using Lemma \ref{lem.Opera}(3) with $\rho = (\alpha+\gamma-1)\frac{2\beta}{\alpha} > 0$ as well as \eqref{inequ.minus1squre} reveals
\begin{align*}
\int_0^s \| A^{\frac{\rho}{2}} ( \cS_{\gamma}(t-u) - \cS_{\gamma}(s-u) ) \|_{\cL(H)}^2 \rd u
\leq C \int_0^s \Big| \int_{s}^{t} (v-u)^{-1} \rd v \ \Big|^2 \rd u 
\leq C(t-s).
\end{align*}

(3) It follows from Lemma \ref{lem.Opera}(3) and \cite[Lemma 3.2]{DaiXiaoBu2021} that
\begin{align*}
&\quad\ \int_0^s \| A^{\frac{\rho}{2}} ( \cS_{\gamma}(t-u) - \cS_{\gamma}(s-u) ) \|_{\cL(H)}^2 \rd u \\
&\leq C \int_0^s \left( (t-u)^{\alpha+\gamma-1-\frac{\alpha}{2\beta}\rho^{+}} - (s-u)^{\alpha+\gamma-1-\frac{\alpha}{2\beta}\rho^{+}} \right)^2 \rd u \\
&\leq
\begin{cases}
C (t-s)^{2-\varepsilon}, & \mbox{if } \rho \geq 0 \mbox{ and } \rho = (\alpha+\gamma-\frac{3}{2})\frac{2\beta}{\alpha}, \\
C (t-s)^{2\min\{\alpha+\gamma-\frac{1}{2}-\frac{\alpha\rho}{2\beta}, 1\}}, & \mbox{if } \rho \geq 0 \mbox{ and } \rho \neq (\alpha+\gamma-\frac{3}{2})\frac{2\beta}{\alpha}, \\
C (t-s)^{2-\varepsilon}, & \mbox{if } \rho < 0 \mbox{ and } \alpha+\gamma = \frac{3}{2}, \\
C (t-s)^{2\min\{\alpha+\gamma-\frac{1}{2}, 1\}}, & \mbox{if } \rho < 0 \mbox{ and } \alpha+\gamma \neq \frac{3}{2}
\end{cases} \\
&\leq C(t-s)^{2\min\{\alpha+\gamma-\frac{1}{2}-\frac{\alpha}{2\beta} \rho^{+}, 1-\varepsilon\}}. 
\end{align*}
The proof is completed. \qed
\end{proof}

\section{The multiplicative noise case}
\label{sec.pfofwellpose}
Consider the stochastic space-time fractional diffusion equation with multiplicative noise
\begin{align} \label{Model multiplicative}
\partial_t^{\alpha} U(t) + A^{\beta} U(t) = F(U(t)) + \mathcal{I}_t^{\gamma}  \big[ \sigma(U(t)) \dot{W}(t) \big], \qquad \forall\, t \in (0,T] 
\end{align} 
with $U(0) = X_0$, where $\sigma$ satisfies the following assumption.

\begin{Ass}\label{asp:Multi}
Let $\alpha \in (0,1)$, $\beta \in (0,1]$, $\gamma \in [0,1]$ with $\alpha + \gamma > \frac{1}{2}$. Let $\sigma: H \rightarrow \mathcal{L}_2(Q^{\frac12}(H),\dot{H}^{r-\kappa})$ for some $r \in (0,\kappa]$ with $\kappa := \min\{ (\alpha+\gamma-\frac{1}{2})\frac{2\beta}{\alpha}, 2\beta \} - \varepsilon_0$ ($\varepsilon_0 > 0$ is arbitrarily small). Moreover, there exists $L>0$ such that
\begin{align*}
& \big\| A^{\frac{r-\kappa}{2}} \sigma(\phi) Q^{\frac{1}{2}} \big\|_{\cL_2(H)} \leq L( 1 + \|\phi\| ), \qquad\quad\ \ \forall\, \phi \in H, \\
& \big\| \big( \sigma(\phi) - \sigma(\psi) \big) Q^{\frac{1}{2}} \big\|_{\mathcal{L}_2(H)} \leq L\|\phi - \psi\|, \qquad\  \forall\, \phi, \psi \in H. 
\end{align*}
\end{Ass}

Notice that if for all $\phi \in H$, $\sigma(\phi) = I$ is the identity operator on $H$, then Assumption \ref{asp:Multi} is equivalent to Assumption \ref{ass.Noise}.
 Hence Theorem \ref{thm.ExisUniq} is a special version of
the following theorem.

\begin{Theo} [existence and uniqueness] \label{thm.ExisUniq-Multi}
Let $X_0 \in L^p(\Omega,H)$ with some $p \geq 2$, and Assumptions \ref{asp:Multi} and \ref{ass.Lip1} hold. Then Eq.\ \eqref{Model multiplicative} admits a unique mild solution $U \in \cC([0,T], L^p(\Omega,H))$ given by 
\begin{align*} 
U(t) = \cS_{1-\alpha}(t) X_0 + \int_0^t \cS_{0}(t-s) F(U(s)) \rd s + \int_0^t \cS_{\gamma}(t-s) \sigma(U(s)) \rd W(s), \quad t \in [0,T].  
\end{align*}
\end{Theo}

\begin{proof}
Let $\varpi > 0$ to be determined and denote $\upsilon := \min\{\alpha,2(\alpha+\gamma-\frac12)\}$. Set 
\begin{align*}
\cC([0,T], L^p(\Omega,H))_{\varpi} := \left\{ \phi \in \cC([0,T], L^p(\Omega,H))\, \Big| \, \| \phi \|_{\varpi,p} := \sup_{t\in[0,T]} \left( \frac{\hE \| \phi (t) \|^p}{E_{\upsilon,1}(\varpi t^{\upsilon})} \right)^{1/p} < \infty \right\}.
\end{align*}
Note that 
it is a Banach space, since the weighted norm $\| \cdot \|_{\varpi,p}$ is equivalent to the standard norm of $\cC([0,T], L^p(\Omega,H))$ for the fixed $\varpi > 0$ and $p \geq 2$.
For any $U\in \cC([0,T], L^p(\Omega,H))_{\varpi}$, define 
\begin{align*}
\cT U(t) := \cS_{1-\alpha}(t) X_0 + \int_0^t \cS_{0}(t-s) F(U(s)) \rd s + \int_0^t \cS_{\gamma}(t-s) \sigma(U(s)) \rd W(s),\quad t\in[0,T].
\end{align*}

\textit{Invariance}:\ Let $U \in \cC([0,T], L^p(\Omega,H))_{\varpi}$ be arbitrary. It follows from the Burkholder--Davis--Gundy inequality (see e.g., \cite{DaPrato2014Book}), Lemma \ref{lem.Opera}(1), Assumptions \ref{asp:Multi} and \ref{ass.Lip1} that
\begin{align*}
\| \cT U(t) \|_{L^p(\Omega,H)}
&\leq \| \cS_{1-\alpha}(t) \|_{\cL(H)} \| X_0 \|_{L^p(\Omega,H)} + \int_0^t \| \cS_{0}(t-s) \|_{\cL(H)} \| F(U(s)) \|_{L^p(\Omega,H)} \rd s \\
&\quad + C \Big( \int_0^t \| A^{\frac{\kappa-r}{2}} \cS_{\gamma}(t-s) \|_{\cL(H)}^2 \| A^{\frac{r-\kappa}{2}} \sigma(U(s)) Q^{\frac{1}{2}} \|_{L^p(\Omega,\mathcal{L}_2(H))}^2 \rd s \Big)^{1/2} \\
&\leq C + C \int_0^t (t-s)^{\alpha-1} \rd s \, \Big( 1 + \| U \|_{\cC([0,T], L^p(\Omega,H))} \Big) \\
&\quad + C \Big( \int_0^t (t-s)^{2(\alpha+\gamma-1 - \frac{\alpha}{2\beta}(\kappa-r))} \rd s \Big)^{1/2} \, \Big( 1 + \| U \|_{\cC([0,T], L^p(\Omega,H))} \Big) \\
&\leq C,\qquad \forall\,t\in[0,T],
\end{align*}
which implies $\cT U \in \cC([0,T], L^p(\Omega,H))_{\varpi}$. This means that 
the space $\cC([0,T], L^p(\Omega,H))_{\varpi}$ is invariant under the mapping $\mathcal T$.

\textit{Contraction}:\ Let $U_1, U_2 \in \cC([0,T], L^p(\Omega,H))_{\varpi}$ be arbitrary. According to the Burkholder--Davis--Gundy inequality, Lemma \ref{lem.Opera}(1), Assumptions \ref{asp:Multi} and \ref{ass.Lip1}, and H\"older's inequality, one can read
\begin{align*}
\hE \| \cT U_1(t) - \cT U_2(t) \|^p
&\leq \Big( \int_0^t \| \cS_{0}(t-s) \|_{\cL(H)} \| F(U_1(s)) - F(U_2(s)) \|_{L^p(\Omega,H)} \rd s \Big)^p \\
&\quad + C \Big( \int_0^t \| \cS_{\gamma}(t-s) \|_{\cL(H)}^2 \| \big( \sigma(U_1(s)) - \sigma(U_2(s)) \big) Q^{\frac{1}{2}} \|_{L^p(\Omega,\mathcal{L}_2(H))}^2 \rd s \Big)^{p/2} \\ 
&\leq C \Big( \int_0^t (t-s)^{\alpha-1} \| U_1(s) - U_2(s) \|_{L^p(\Omega,H)} \rd s \Big)^p \\
&\quad + C \Big( \int_0^t (t-s)^{2(\alpha+\gamma-1)} \| U_1(s) - U_2(s) \|_{L^p(\Omega,H)}^2 \rd s \Big)^{p/2} \\ 
&\leq C_0 \int_0^t (t-s)^{\upsilon-1} \| U_1(s) - U_2(s) \|_{L^p(\Omega,H)}^p \rd s. 
\end{align*}
Then, it follows from $\frac{\varpi}{\Gamma(\upsilon)}\int_0^t (t-s)^{\upsilon-1} E_{\upsilon,1}(\varpi s^{\upsilon}) \rd s\le E_{\upsilon,1}(\varpi t^{\upsilon})$ that 
\begin{align*}
\frac{\hE \| \cT U_1(t) - \cT U_2(t) \|^p}{E_{\upsilon,1}(\varpi t^{\upsilon})}
&\leq \frac{C_0}{E_{\upsilon,1}(\varpi t^{\upsilon})} \int_0^t (t-s)^{\upsilon-1} E_{\upsilon,1}(\varpi s^{\upsilon}) \rd s \, \| U_1 - U_2 \|_{\varpi,p}^p \\ 
&\leq \frac{C_0 \Gamma(\upsilon)}{\varpi} \| U_1 - U_2 \|_{\varpi,p}^p,\qquad \forall\,t\in[0,T].
\end{align*}
Since the positive constant $C_0$ is independent of $\varpi$, one can take $\varpi=2^pC_0 \Gamma(\upsilon)$ such that 
$\| \cT U_1 - \cT U_2 \|_{\varpi,p} \leq \frac{1}{2} \| U_1 - U_2 \|_{\varpi,p},$
and the mapping $\mathcal T$ is therefore a contraction on $\cC([0,T], L^p(\Omega,H))_{\varpi}$. 

As a result, the mapping $\mathcal T$ admits a unique fixed point in $\cC([0,T], L^p(\Omega,H))$, so the model \eqref{Model multiplicative} admits a unique mild solution $U \in \cC([0,T], L^p(\Omega,H))$. The proof is completed.
\qed
\end{proof}

\begin{Prop} \label{prop.Regu-Multi}
Let $p \geq 2$, and Assumptions \ref{asp:Multi} and \ref{ass.Lip1} hold.
\begin{enumerate}
\item[(1)] If $X_0 \in L^p(\Omega,\dot{H}^{r})$, then $U \in \cC([0,T], L^p(\Omega,\dot{H}^{r}))$.

\item[(2)] If $X_0 \in L^p(\Omega,\dot{H}^{2\beta})$, then there exists some positive constant $C$ such that
\begin{align*}
\| U(t) - U(s) \|_{L^p(\Omega,H)}
\leq C (t-s)^{ \min\{ \alpha, \frac{\alpha r}{2\beta} + (\gamma-\frac{1}{2})^{+} \} }, \qquad \forall \, 0\leq s < t \leq T.
\end{align*}

\item[(3)] If $X_0 \in L^p(\Omega,\dot{H}^{2\beta})$, then there exists some positive constant $C_{\varepsilon}$ such that
\begin{align*}
\| U(t) - U(s) \|_{L^p(\Omega,H)}
\leq C_{\varepsilon} s^{\varepsilon-\frac{1}{2}} (t-s)^{\min\{ \frac{\alpha r}{2\beta} + (\gamma-\frac{1}{2})^{+}, 1-\varepsilon \}}, \quad\ \ \forall\, 0 < s < t \leq T.
\end{align*}
\end{enumerate}
\end{Prop}

\begin{proof}
For the multiplicative noise case, we redefine the stochastic convolution 
\begin{align*} 
\Lambda_{\sigma\circ U}(t) := \int_0^{t} \cS_{\gamma}(t-s) \sigma(U(s)) \rd W(s), \qquad \forall\, t \in [0,T].
\end{align*}
Then 
$U(t) = \cS_{1-\alpha}(t) X_0 +\Upsilon_{F\circ U}(t) +\Lambda_{\sigma\circ U}(t)$. In terms of the proof of the properties (1)--(3),
 the estimate of $\cS_{1-\alpha}(t) X_0 +\Upsilon_{F\circ U}(t) $ are the same as that of $\cS_{1-\alpha}(t) X_0 +\Upsilon_{F\circ X}(t) $
  in the proofs of Propositions \ref{prop.Regu} and \ref{th.sinRegu}. Hence it suffices to prove the properties (1)--(3) for the special case that $U(t)=\Lambda_{\sigma\circ U}(t)$ (i.e., $X_0=0$ and $F\equiv0$).

By the Burkholder--Davis--Gundy inequality, Lemma \ref{lem.Opera}(1) and Assumption \ref{asp:Multi}, 
\begin{align*}
\| U(t) \|^2_{L^p(\Omega,\dot{H}^{r})}&=\| \Lambda_{\sigma\circ U}(t) \|^2_{L^p(\Omega,\dot{H}^{r})}\\
&\leq  C \int_0^t \| A^{\frac{\kappa}{2}} \cS_{\gamma}(t-s) \|_{\cL(H)}^2 \| A^{\frac{r-\kappa}{2}} \sigma(U(s)) Q^{\frac{1}{2}} \|_{L^p(\Omega,\mathcal{L}_2(H))}^2 \rd s  \\
&\leq C \int_0^t (t-s)^{ 2(\alpha+\gamma-1-\frac{\alpha \kappa}{2\beta}) } \big(1 + \|U(s)\|_{L^p(\Omega,H)}^2\big) \rd s , \qquad \forall\, t \in [0,T], 
\end{align*}
in which the norm $\| \cdot \|_{L^p(\Omega,H)} $ can be further bounded by $C \| \cdot \|_{L^p(\Omega,\dot{H}^{r})}$. 
Thus the singular-type Gr\"onwall's inequality yields the spatial regularity result $U \in \cC([0,T], L^p(\Omega,\dot{H}^{r}))$.

When $U(t)=\Lambda_{\sigma\circ U}(t)$, the properties (2) and (3) can be obtained by the following inequality
\begin{align} \label{eq:GGG} 
\| \Lambda_{\sigma\circ U}(t) - \Lambda_{\sigma\circ U}(s) \|_{L^p(\Omega,H)} 
\leq C (t-s)^{ \min\{ 1-\varepsilon, \frac{\alpha r}{2\beta} + (\gamma-\frac{1}{2})^{+} \} }, \quad\ \ \forall\, 0 < s < t \leq T.
\end{align}
Therefore it suffices to prove \eqref{eq:GGG}. Indeed,
similarly to the proof of Proposition \ref{prop.Lambda}, using the Burkholder--Davis--Gundy inequality, Lemma \ref{lem.Opera}, Assumption \ref{asp:Multi} and Theorem \ref{thm.ExisUniq-Multi} shows 
\begin{align*} 
 \| \Lambda_{\sigma\circ U}(t) - \Lambda_{\sigma\circ U}(s) \|_{L^p(\Omega,H)}^2 
&\leq C \Big\| \int_s^{t} \cS_{\gamma}(t-\tau) \sigma(U(\tau)) \rd W(\tau) \Big\|_{L^p(\Omega,H)}^2 \\
&\quad + C \Big\| \int_0^{s} \big( \cS_{\gamma}(t-\tau) - \cS_{\gamma}(s-\tau) \big) \sigma(U(\tau)) \rd W(\tau) \Big\|_{L^p(\Omega,H)}^2 \\ 
&\leq C \int_{s}^t \big\| \cS_{\gamma}(t-\tau) \sigma(U(\tau)) Q^{\frac{1}{2}} \big\|_{L^p(\Omega,\cL_2(H))}^2 \rd \tau \\
&\quad + C \int_0^{s} \big\| \big( \cS_{\gamma}(t-\tau) - \cS_{\gamma}(s-\tau) \big) \sigma(U(\tau)) Q^{\frac{1}{2}} \big\|_{L^p(\Omega,\cL_2(H))}^2 \rd \tau \\
&\leq C \int_{s}^t \big\| A^{\frac{\kappa-r}{2}} \cS_{\gamma}(t-\tau) \big\|_{\cL(H)}^2 
\big\| A^{\frac{r-\kappa}{2}} \sigma(U(\tau)) Q^{\frac{1}{2}} \big\|_{L^p(\Omega,\cL_2(H))}^2 \rd \tau \\ 
&\quad + C \int_0^{s} \big\| A^{\frac{\kappa-r}{2}} \big( \cS_{\gamma}(t-\tau) - \cS_{\gamma}(s-\tau) \big) \big\|_{\cL(H)}^2\big\| A^{\frac{r-\kappa}{2}} \sigma(U(\tau)) Q^{\frac{1}{2}} \big\|_{L^p(\Omega,\cL_2(H))}^2 \rd \tau \\
&\leq C (t-s)^{2( \alpha + \gamma - \frac{1}{2} - \frac{\alpha}{2\beta} (\kappa- r)) } + 
\begin{cases}
C (t-s)^{2(\frac{1}{2} - \frac{\alpha}{2\beta}(\kappa- r)) }, & \mbox{if } \alpha + \gamma = 1, \\
C (t-s)^{2\min\{\alpha+\gamma-\frac{1}{2}-\frac{\alpha}{2\beta} (\kappa-r), 1-\varepsilon\}}, & \mbox{if } \alpha + \gamma \neq 1
\end{cases}\\
&\leq C (t-s)^{ 2\min\{ \frac{\alpha r}{2\beta} + (\gamma-\frac{1}{2})^{+}, 1-\varepsilon \} }, 
\end{align*}
as required. The proof is completed.
\qed
\end{proof}

\end{appendices}

\noindent \textbf{Funding} This work is supported by National key R\&D Program of China (No.\ 2020YFA0713701), National Natural Science Foundation of China (Nos.\ 11971470, 12031020), and China Postdoctoral Science Foundation (No.\ 2022M713313).

\pdfbookmark[1]{References}{anchor}
\bibliographystyle{siam}
\bibliography{RefBib}

\begin{thebibliography}{10}

\bibitem{AlloubaXiao2017}
{\sc H.~Allouba and Y.~Xiao}, {\em L-{K}uramoto-{S}ivashinsky {SPDE}s vs.
  time-fractional {SPIDE}s: exact continuity and gradient moduli,
  {$1/2$}-derivative criticality, and laws}, J. Differential Equations, 263
  (2017), pp.~1552--1610.

\bibitem{BD00}
{\sc A.~Budhiraja and P.~Dupuis}, {\em A variational representation for
  positive functionals of infinite dimensional {B}rownian motion}, Probab.
  Math. Statist., 20 (2000), pp.~39--61.

\bibitem{BDM08}
{\sc A.~Budhiraja, P.~Dupuis, and V.~Maroulas}, {\em Large deviations for
  infinite dimensional stochastic dynamical systems}, Ann. Probab., 36 (2008),
  pp.~1390--1420.

\bibitem{CHJS21}
{\sc C.~Chen, J.~Hong, D.~Jin, and L.~Sun}, {\em Asymptotically-preserving
  large deviations principles by stochastic symplectic methods for a linear
  stochastic oscillator}, SIAM J. Numer. Anal., 59 (2021), pp.~32--59.

\bibitem{CHJS22}
{\sc C.~Chen, J.~Hong, D.~Jin, and L.~Sun}, {\em Large deviations principles
  for symplectic discretizations of stochastic linear {S}chr\"odinger
  equation}, Potential Anal.,~DOI:\
  \href{https://doi.org/10.1007/s11118-022-09990-z}{10.1007/s11118-022-09990-z},
   (2022).

\bibitem{ChenHu2022}
{\sc L.~Chen and G.~Hu}, {\em H\"{o}lder regularity for the nonlinear
  stochastic time-fractional slow \& fast diffusion equations on {$\Bbb R^d$}},
  Fract. Calc. Appl. Anal., 25 (2022), pp.~608--629.

\bibitem{ChenHu2019}
{\sc L.~Chen, Y.~Hu, and D.~Nualart}, {\em Nonlinear stochastic time-fractional
  slow and fast diffusion equations on {$\Bbb R^d$}}, Stochastic Process.
  Appl., 129 (2019), pp.~5073--5112.

\bibitem{Chen2015SPA}
{\sc Z.-Q. Chen, K.-H. Kim, and P.~Kim}, {\em Fractional time stochastic
  partial differential equations}, Stochastic Process. Appl., 125 (2015),
  pp.~1470--1499.

\bibitem{CS05}
{\sc Z.-Q. Chen and R.~Song}, {\em Two-sided eigenvalue estimates for
  subordinate processes in domains}, J. Funct. Anal., 226 (2005), pp.~90--113.

\bibitem{DaPrato2014Book}
{\sc G.~Da~Prato and J.~Zabczyk}, {\em Stochastic Equations in Infinite
  Dimensions}, Cambridge University Press, Cambridge, second~ed., 2014.

\bibitem{DaiXiaoBu2021}
{\sc X.~Dai, A.~Xiao, and W.~Bu}, {\em Stochastic fractional
  integro-differential equations with weakly singular kernels: well-posedness
  and {E}uler--{M}aruyama approximation}, Discrete Contin. Dyn. Syst. Ser. B,
  27 (2022), pp.~4231--4253.

\bibitem{DG93}
{\sc G.~Dal~Maso}, {\em An {I}ntroduction to {$\Gamma$}-Convergence},
  Birkh\"{a}user Boston, Inc., Boston, MA, 1993.

\bibitem{Garrappa2013}
{\sc R.~Garrappa}, {\em Exponential integrators for time-fractional partial
  differential equations}, Eur. Phys. J. Special Topics, 222 (2013),
  pp.~1915--1927.

\bibitem{Gunzburger2019MC}
{\sc M.~Gunzburger, B.~Li, and J.~Wang}, {\em Sharp convergence rates of time
  discretization for stochastic time-fractional {PDE}s subject to additive
  space-time white noise}, Math. Comp., 88 (2019), pp.~1715--1741.

\bibitem{HuLi2022ANM}
{\sc Y.~Hu, C.~Li, and Y.~Yan}, {\em Weak convergence of the {L}1 scheme for a
  stochastic subdiffusion problem driven by fractionally integrated additive
  noise}, Appl. Numer. Math., 178 (2022), pp.~192--215.

\bibitem{Jentzen2009ProcA}
{\sc A.~Jentzen and P.~E. Kloeden}, {\em Overcoming the order barrier in the
  numerical approximation of stochastic partial differential equations with
  additive space-time noise}, Proc. R. Soc. Lond. Ser. A Math. Phys. Eng. Sci.,
  465 (2009), pp.~649--667.

\bibitem{Jin2019ESAIM}
{\sc B.~Jin, Y.~Yan, and Z.~Zhou}, {\em Numerical approximation of stochastic
  time-fractional diffusion}, ESAIM Math. Model. Numer. Anal., 53 (2019),
  pp.~1245--1268.

\bibitem{Kang2021IMA}
{\sc W.~Kang, B.~A. Egwu, Y.~Yan, and A.~K. Pani}, {\em Galerkin finite element
  approximation of a stochastic semilinear fractional subdiffusion with
  fractionally integrated additive noise}, IMA J. Numer. Anal., 42 (2022),
  pp.~2301--2335.

\bibitem{Kovacs2020SIAM}
{\sc M.~Kov\'{a}cs, S.~Larsson, and F.~Saedpanah}, {\em Mittag--{L}effler
  {E}uler integrator for a stochastic fractional order equation with additive
  noise}, SIAM J. Numer. Anal., 58 (2020), pp.~66--85.

\bibitem{KD01}
{\sc H.~J. Kushner and P.~Dupuis}, {\em Numerical Methods for Stochastic
  Control Problems in Continuous Time}, Springer-Verlag, New York, second~ed.,
  2001.

\bibitem{MijenaNane2015}
{\sc J.~B. Mijena and E.~Nane}, {\em Space-time fractional stochastic partial
  differential equations}, Stochastic Process. Appl., 125 (2015),
  pp.~3301--3326.

\bibitem{NieDeng2022}
{\sc D.~Nie and W.~Deng}, {\em A unified convergence analysis for the
  fractional diffusion equation driven by fractional {G}aussian noise with
  {H}urst index {$H\in(0,1)$}}, SIAM J. Numer. Anal., 60 (2022),
  pp.~1548--1573.

\bibitem{Podlubny1999}
{\sc I.~Podlubny}, {\em Fractional Differential Equations:\ An Introduction to
  Fractional Derivatives, Fractional Differential Equations, to Methods of
  Their Solution and Some of Their Applications}, Academic Press Inc., San
  Diego, 1999.

\bibitem{RY99}
{\sc D.~Revuz and M.~Yor}, {\em Continuous Martingales and {B}rownian Motion},
  Springer-Verlag, Berlin, third~ed., 1999.

\bibitem{WangQi2015}
{\sc X.~Wang and R.~Qi}, {\em A note on an accelerated exponential {E}uler
  method for parabolic {SPDE}s with additive noise}, Appl. Math. Lett., 46
  (2015), pp.~31--37.

\bibitem{WuYan2020ANM}
{\sc X.~Wu, Y.~Yan, and Y.~Yan}, {\em An analysis of the {L}1 scheme for
  stochastic subdiffusion problem driven by integrated space-time white noise},
  Appl. Numer. Math., 157 (2020), pp.~69--87.

\bibitem{YanYin2018}
{\sc L.~Yan and X.~Yin}, {\em Large deviation principle for a space-time
  fractional stochastic heat equation with fractional noise}, Fract. Calc.
  Appl. Anal., 21 (2018), pp.~462--485.

\bibitem{Zhang2008JDE}
{\sc X.~Zhang}, {\em Euler schemes and large deviations for stochastic
  {V}olterra equations with singular kernels}, J. Differential Equations, 244
  (2008), pp.~2226--2250.

\end{thebibliography}

%

\end{document}